\documentclass[11pt]{article} 

\usepackage[T1]{fontenc}
\usepackage[utf8]{inputenc}
\usepackage{amsmath}
\usepackage{amssymb}
\usepackage{amsthm}
\usepackage{amsfonts}
\usepackage{bbm}
\usepackage[a4paper,margin=3cm]{geometry}
\usepackage[hidelinks]{hyperref}
\usepackage{lmodern}
\usepackage{microtype}

\newtheorem{definition}{Definition}[section]
\newtheorem{theorem}[definition]{Theorem}
\newtheorem{proposition}[definition]{Proposition}
\newtheorem{corollary}[definition]{Corollary}
\newtheorem{lemma}[definition]{Lemma}
\newtheorem{remark}[definition]{Remark}
\theoremstyle{definition}
\newtheorem{example}[definition]{Example}

\newcommand*\diff{\mathop{}\!\mathrm{d}}
\newcommand{\intco}[2]{\int\limits_{[#1,#2)}}
\newcommand{\stirlingone}[2]{\genfrac{[}{]}{0pt}{}{#1}{#2}}
\newcommand{\stirlingtwo}[2]{\genfrac{\lbrace}{\rbrace}{0pt}{}{#1}{#2}}

\renewcommand{\phi}{\varphi}

\DeclareMathOperator{\ZZ}{\mathbb{Z}}
\DeclareMathOperator{\NN}{\mathbb{N}}
\DeclareMathOperator{\RR}{\mathbb{R}}
\DeclareMathOperator{\idop}{\mathbbm{1}}
\DeclareMathOperator{\Leb}{Leb}
\DeclareMathOperator{\PP}{\mathbf{P}}
\DeclareMathOperator{\EE}{\mathbf{E}}
\DeclareMathOperator{\LL}{\mathcal{L}}
\DeclareMathOperator{\GG}{\mathcal{G}\!}
\DeclareMathOperator{\ee}{\mathrm{e}}
\title{On occupation times of one-dimensional diffusions}

\author{Paavo Salminen \qquad David Stenlund \\[0.5em]
	{\small \AA bo Akademi University} \\
	{\small Faculty of Science and Engineering} \\
	{\small FI-20500 \AA bo, Finland} \\
	{\small \ttfamily phsalmin@abo.fi, david.stenlund@abo.fi}
}

\date{\today}

\begin{document}

\maketitle

\begin{abstract}
In this paper we study the moment generating function and the moments of occupation time functionals of one-dimensional diffusions. Assuming, specifically, that the process lives on $\RR$ and starts at~0, we apply Kac's moment formula and the strong Markov property to derive an expression for the moment generating function in terms of the Green kernel of the underlying diffusion. Moreover, the approach allows us to derive a recursive equation for the Laplace transforms of the moments of the occupation time on $\RR_+$. If the diffusion has a scaling property, the recursive equation simplifies to an equation for the moments of the occupation time up to time 1. As examples of diffusions with scaling property we study in detail skew two-sided Bessel processes and, as a special case, skew Brownian motion. It is seen that for these processes our approach leads to simple explicit formulas. The recursive equation for a sticky Brownian motion is also discussed.
 \\ \\
Keywords: additive functional, arcsine law, Green kernel, oscillating Brownian motion, Brownian spider, Lamperti distribution
\\ \\
AMS Classification: 60J60, 60J55, 05A10, 60J65
\end{abstract}

\newpage
\tableofcontents
\newpage

\section{Introduction}

To determine the distribution of the time a stochastic process spends in a set (up to a given time) is a classical, much studied and well understood problem. For diffusions the Feynman--Kac formula is the basic tool to attack the problem. If the generator of the diffusion has smooth coefficients this approach calls for solving a parabolic differential equation with a boundary condition. See, e.g., for Brownian motion Durrett~\cite[Section~4]{Durrett}, and for diffusions Karatzas and Shreve~\cite[Section~5.7]{Karatzas}. We refer also to Borodin and Salminen~\cite{handbook} for explicit examples of various one-dimensional diffusions. 

Undoubtedly, the most famous occupation time distribution is the arcsine law, which dates back to L\'{e}vy~\cite{Levy}. According to this, for a standard Brownian motion~$W$, letting $A^W_1$ denote the time $W$ is positive up to time $1$, it holds that
\begin{equation}\label{eq_arcsine_law}
\PP_0\left(A_1^W \leq x\right) = \frac{2}{\pi} \arcsin(\sqrt{x}), \quad x \in [0,1],
\end{equation} 
where $\PP_0$ is the probability measure associated with $W$ (when initiated at~0).
For proofs of~\eqref{eq_arcsine_law} based on the Feynman--Kac formula, see~\cite[p.~273]{Karatzas} and M\"{o}rters and Peres~\cite[p.~207]{Moerters}. 
In the latter one, instead of inverting the double Laplace transform, the moments of the distribution are calculated, and that approach is, in a sense, closer to the one discussed in this paper. 

We are interested in the occupation times for a regular one-dimensional diffusion. More precisely, let $(X_t)_{t\geq 0}$ be such a diffusion and introduce, for $t\geq 0$, 
\begin{equation*}
A_t^X := \Leb \{s\in[0,t]:X_s\geq 0\},
\end{equation*} 
where $\Leb$ denotes the Lebesgue measure. Recall that in~\cite{Truman} Truman and Williams calculated, applying the Feynman-Kac method, a formula for the moment generating function of the occupation time up to an independent exponential time for a fairly general (positively recurrent) diffusion. Watanabe~\cite{Watanabe} extended the result to hold for a general regular (gap) diffusion exploiting the random time change techniques. However, earlier Barlow, Pitman and Yor~\cite{Barlow} derived the formula in case of a skew two sided Bessel process using excursion theory. This result was connected in~\cite{Watanabe} with a distribution found by Lamperti in~\cite{Lamperti}. In~\cite{PY} Pitman and Yor proved, and also extended, the general formula presented in~\cite{Watanabe} using the excursion theory (but they also discuss an approach via the Feynman-Kac method). We refer also to Watanabe, Yano and Yano~\cite{WYY}, where the inversion of the Laplace transform is discussed, and to Kasahara and Yano~\cite{KasaharaYano} for results of the asymptotic behavior of the density at~0. 

Our main contributions are, firstly, a new expression for the moment generating function of the occupation time up to an exponential time, as well as a recursive equation for the Laplace transforms of the moments of the occupation time. A novel feature in our analysis is, perhaps, that it is based explicitly on Kac's moment formula and not on the Feynman--Kac formula. In spite of the fact that both formulas have been known for decades we have not been able to find precisely these results in the literature. It is seen that the general formula in~\cite{Watanabe} can be obtained from our formula via some straightforward calculations. Secondly, the recursive formula for the moments is solved for skew two-sided Bessel process. Somewhat surprisingly, the result says that the moments are polynomials in the dimension and skewness parameters. Although the density of the occupation time is known in this case, it does not seem to be possible to find the general formula for the moments via integration, but numerical integration can, of course, be performed to check the formula for some given values of the parameters. Skew Brownian motion is a special case of a skew two-sided Bessel process obtained when the dimension parameter is $-1/2$, and in this case the formula for the moments is simpler. 

The paper is organized as follows. In the next section some basic ingredients from the theory of diffusions are recalled. We also introduce the diffusions for which the occupation times are studied later in the paper. In Section~\ref{section_Kac} we discuss Kac's moment formula for integral functionals. Although this formula is well known a short proof is included for completeness of the presentation. Section~\ref{MOM} contains the formula for the moment generating function, see~\eqref{eq_thmE0}, and it is also proved that this coincides with the formula in~\cite{Watanabe}. In Section~\ref{section_moments} we derive the recursive equation for the Laplace transforms of the moments of the occupation time, see Theorem~\ref{thm_rec}. In the proof a technical result, Lemma~\ref{lemma_f0}, is needed, the proof of which is given in Appendix. In Section~\ref{section_examples} we apply the results on a number of diffusions and present formulas for the skew two-sided Bessel process, skew Brownian motion, oscillating Brownian motion, Brownian spider and sticky Brownian motion.

\section{Preliminaries on diffusions}
\label{section_prelim}

Let $X=(X_t)_{t\geq 0}$ be a regular diffusion taking values on an interval $I\subseteq\RR$. For simplicity, it is assumed that $X$ is conservative, i.e., $\PP_x(X_t\in I)=1$ for all~$x\in I$ and $t\geq 0,$ where $\PP_x$ stands for the probability measure associated with~$X$ when initiated at~$x$. To fix ideas, we suppose that $0\in I$. In this section we briefly describe the setup for the diffusion~$X$. 

The notations $m$ and $S$ are used for the speed measure and the scale function, respectively. It is assumed that $S$ is normalized to satisfy $S(0)=0$. Recall that $X$ has a transition density~$p$ with respect to~$m$, that is, for any Borel subset~$B$ of~$I$ it holds that
\begin{equation*}
\PP_x\left(X_t\in B\right) = \int_B p(t;x,y)\,m(dy).
\end{equation*}
Moreover, for~$\lambda> 0$ let $\varphi_\lambda$ and $\psi_\lambda$ denote the decreasing and increasing, respectively, positive and continuous solutions of the generalized ODE (see~\cite[p.~18]{handbook}) 
\begin{equation*}
\frac{d}{dm}\frac{d}{dS} u=\lambda u.
\end{equation*} 
The solutions 
 $\psi_\lambda $ and $\varphi_\lambda $ are unique up to a multiplicative constant when appropriate boundary conditions are imposed. The Wronskian constant $\omega_{\lambda}$ is defined as 
\begin{align}\label{eq_wronskian}
\omega_{\lambda}&:=\psi^{+}_{\lambda}(x)\varphi_\lambda (x)-\psi_\lambda (x)\varphi^{+}_\lambda (x) \nonumber\\
&\:=\psi^{-}_{\lambda }(x)\varphi_\lambda (x)-\psi_\lambda (x)\varphi^{-}_{\lambda }(x),
\end{align} 
where the superscripts $^+$ and $^-$ denote the right and left derivatives with respect to the scale function~$S$. 
We remark that in case $x$ is not a sticky point it holds that $\psi^{+}_{\lambda}(x) = \psi^{-}_{\lambda}(x)$ and $\phi^{+}_{\lambda}(x) = \phi^{-}_{\lambda}(x)$ \cite[p.~129]{Ito} (cf.~\cite[Thm.~3.12]{Revuz}). 
The Green kernel (also called the resolvent kernel) \cite[p.~150]{Ito} is given by 
\begin{equation}\label{eq_greenkernel}
G_\lambda (x,y):=
\begin{cases}
w^{-1}_\lambda \psi_\lambda (x)\varphi_\lambda (y), & x\leq y,\\
w^{-1}_\lambda \psi_\lambda (y)\varphi_\lambda (x), & x\geq y,
\end{cases}
\end{equation}
and satisfies
\begin{equation*}
G_\lambda (x,y)=\int_0^\infty \ee^{-\lambda t}p(t;x,y)\, \diff t.
\end{equation*}
It is well known that the first hitting time $H_y:=\inf\{t\geq 0 : X_t=y\}$ has a density with respect to the Lebesgue measure. Especially for~$H_0$ we use the notation
\begin{equation*}
f(x;t) := \PP_x(H_0\in \diff t)/\diff t
\end{equation*}
for the $\PP_x$-density of~$H_0$ and
\begin{equation}\label{eq_Lhitting}
\widehat{f}(x;\lambda) := \int_0^\infty \ee^{-\lambda t}\PP_x(H_0\in \diff t) = \EE_x(\ee^{-\lambda H_0}) =
\begin{cases}
\dfrac{\varphi_\lambda (x)}{\varphi_\lambda (0)}, & x\geq 0, \\[1em]
\dfrac{\psi_\lambda (x)}{\psi_\lambda (0)}, & x\leq 0,
\end{cases}
\end{equation}
for its Laplace transform. Moreover, let
\begin{equation*}
\widehat{f}_\lambda^{(k)}(x;\lambda) := \frac{\diff^k}{\diff \lambda^k}\widehat{f}(x;\lambda) = (-1)^k \EE_x(H_0^k \ee^{-\lambda H_0}).
\end{equation*}
Using the continuity of~$\phi_\lambda$ it follows from~\eqref{eq_Lhitting} that $\lim_{x\downarrow 0}\widehat{f}(x;\lambda)=1$. We also need the following result regarding the $k$th derivative of~$\widehat{f}(x;\lambda)$, which is perhaps known, but we do not have any reference. 
\begin{lemma}\label{lemma_f0}
For any $\lambda >0$ and $k\geq 1$, 
\begin{equation}\label{eq_f0}
\lim_{x\downarrow 0} \widehat{f}_\lambda^{(k)}(x;\lambda) = 0.
\end{equation}
\end{lemma}
\begin{proof}
A proof based on spectral representations is given in Appendix. 
\hfill\end{proof}

Let $t>0$ and consider the occupation time on $\RR_+:=[0,+\infty)$ up to time $t$: 
\begin{equation}\label{def_At}
A_t^X := \Leb \{s\in[0,t]:X_s\geq 0\}=\int_0^t \idop_{[0,\infty)}(X_s)\,\diff s.
\end{equation}
If~$X$ is a self-similar process (see~\cite{Sato} p.~70), that is, for any $a\geq 0$ there exists $b\geq0$ such that
\begin{equation}\label{eq_scaling}
(X_{at})_{t\geq 0} \overset{(d)}{=} (b X_t)_{t\geq 0},
\end{equation}
then, as is easily seen, for any fixed $t\geq 0$,
\begin{equation}\label{scalingA}
A_t^X \overset{(d)}{=} t A_1^X.
\end{equation}

\begin{example}\label{ex_bessel}
\emph{Skew two-sided Bessel processes} were introduced in~\cite{Barlow}; see also~\cite{Alili,Blei,Watanabe}. A skew two-sided Bessel process $(X^{(\nu,\beta)}_t)_{t\geq 0}$ with parameter $\nu \in (-1,0)$ and skewness parameter $\beta\in(0,1)$ is a diffusion on $\RR$ with the speed measure
\begin{equation*}
m_\nu(\diff x) = 
\begin{cases}
4\beta x^{2\nu+1} \diff x, & x>0, \\
4(1-\beta) |x|^{2\nu+1} \diff x, & x<0,
\end{cases}
\end{equation*}
and the scale function
\begin{equation*}
S_\nu(x) = 
\begin{cases}
-\frac{1}{4\beta\nu} x^{-2\nu}, & x\geq0, \\
\frac{1}{4(1-\beta)\nu} |x|^{-2\nu}, & x\leq0.
\end{cases}
\end{equation*}
The generator is given by
\begin{equation*}
\GG f(x) = \frac{1}{2} f''(x) + \frac{2\nu+1}{2x} f'(x), \; x\neq 0, \quad \GG f(0) = \GG f(0+) = \GG f(0-),
\end{equation*}
with the domain 
\begin{equation*}
\mathcal{D} = \{ f : f, \GG f \in \mathcal{C}_b(\RR) , \frac{\diff f}{\diff S}(0+) = \frac{\diff f}{\diff S}(0-) \}.
\end{equation*}

Recall that for a ``one-sided'' Bessel diffusion on $\RR_+$ reflected at zero \cite[p.~137]{handbook} we have the fundamental solutions
\begin{equation*}
\widehat{\psi}_\lambda(x) = x^{-\nu} I_\nu(x\sqrt{2\lambda}), \quad
\widehat{\phi}_\lambda(x) = x^{-\nu} K_\nu(x\sqrt{2\lambda}), 
\end{equation*}
for~$x>0$. The limits
\begin{equation*}
\widehat{\psi}_\lambda(0) 
= \biggl( \frac{\sqrt{2\lambda}}{2} \biggr)^\nu \frac{1}{\Gamma(1+\nu)}, \quad
\widehat{\phi}_\lambda(0) 
= \biggl( \frac{\sqrt{2\lambda}}{2} \biggr)^\nu \frac{\Gamma(-\nu)}{2}
\end{equation*}
follow from the fact that, for~$\nu\in(-1,0)$ and when $x\rightarrow 0$,
\begin{equation}\label{eq_besselzero}
I_\nu(x) \simeq \frac{1}{\Gamma(\nu+1)} \biggl( \frac{x}{2} \biggr)^\nu, \quad
K_\nu(x) \simeq \frac{\Gamma(-\nu)}{2} \biggl( \frac{x}{2} \biggr)^\nu.
\end{equation}
Let
\begin{equation*}
\psi_\lambda(x) := 
\begin{cases}
\dfrac{\pi}{2\beta\sin (-\pi\nu)}\,\widehat{\psi}_\lambda(x) - \dfrac{1-\beta}{\beta} \,\widehat{\phi}_\lambda(x), & x\geq 0, \\
\widehat{\phi}_\lambda(|x|), & x\leq 0,
\end{cases}
\end{equation*}
and 
\begin{equation*}
\phi_\lambda(x) := 
\begin{cases}
\widehat{\phi}_\lambda(x), & x\geq 0, \\
\dfrac{\pi}{2(1-\beta)\sin (-\pi\nu)}\,\widehat{\psi}_\lambda(|x|) - \dfrac{\beta}{1-\beta} \,\widehat{\phi}_\lambda(|x|), & x\leq 0.
\end{cases}
\end{equation*}
Since $\widehat{\psi}_\lambda$ and $\widehat{\phi}_\lambda$ are solutions of $\GG f(x) = \lambda f(x)$ on $\RR_+$, we immediately see that $\psi_\lambda$ is also a solution when $x>0$, while for~$x<0$
\begin{equation*}
\psi_\lambda^{\,\prime}(x) = -\widehat{\phi}_\lambda^{\,\prime}(|x|), \quad \psi_\lambda^{\,\prime\prime}(x) = \widehat{\phi}_\lambda^{\,\prime\prime\!}(|x|),
\end{equation*}
and, hence,
\begin{equation*}
\GG \psi_\lambda(x) 
= \frac{1}{2} \widehat{\phi}_\lambda^{\,\prime\prime\!}(|x|) + \frac{2\nu+1}{2|x|} \widehat{\phi}_\lambda^{\,\prime}(|x|) 
= \lambda \widehat{\phi}_\lambda (|x|) = \lambda \psi_\lambda(x). 
\end{equation*}
Thus, $\psi_\lambda$ is a solution of the equation
\begin{equation}\label{eq_GfLf}
\GG f(x) = \lambda f(x), \quad x\neq0.
\end{equation}
Furthermore, $\psi_\lambda$ is continuous since 
\begin{align*}
\lim_{x\downarrow 0} \psi_\lambda(x) &= \frac{\Gamma(-\nu)\Gamma(1+\nu)}{2\beta} \widehat{\psi}_\lambda(0) - \frac{\beta}{1-\beta} \widehat{\phi}_\lambda(0) \\
&= \biggl( \frac{\sqrt{2\lambda}}{2} \biggr)^\nu \frac{\Gamma(-\nu)}{2} = \widehat{\phi}_\lambda(0) = \lim_{x\uparrow 0} \psi_\lambda(x),
\end{align*}
using Euler's reflection formula. Notice that from 
\begin{equation*}
\widehat{\psi}_\lambda^{\,\prime}(x) = x^{-\nu} I_{\nu+1}(x\sqrt{2\lambda}), \quad \widehat{\phi}_\lambda^{\,\prime}(x) = - x^{-\nu} K_{\nu+1}(x\sqrt{2\lambda}),
\end{equation*}
and~\eqref{eq_besselzero} it follows that
\begin{equation*}
\lim_{x\downarrow 0} \; x^{2\nu+1} \widehat{\psi}_\lambda^{\,\prime}(x) = 0, \qquad
\lim_{x\downarrow 0} \; x^{2\nu+1}\widehat{\phi}_\lambda^{\,\prime}(x) = - \biggl( \frac{2}{\sqrt{2\lambda}} \biggr)^\nu \Gamma(1+\nu), 
\end{equation*}
and, therefore, 
\begin{align*}
\lim_{x\downarrow 0} \frac{\diff \psi_\lambda(x)}{\diff S} 
&= \frac{\pi}{\sin (-\pi\nu)} \lim_{x\downarrow 0} x^{2v+1}\widehat{\psi}_\lambda^{\,\prime}(x) - 2(1-\beta) \lim_{x\downarrow 0}x^{2v+1}\widehat{\phi}_\lambda^{\,\prime}(x) \\
&= - 2(1-\beta) \lim_{x\uparrow 0} |x|^{2v+1}\widehat{\phi}_\lambda^{\,\prime}(|x|) \\
&= \lim_{x\uparrow 0} \frac{\diff \psi_\lambda(x)}{\diff S},
\end{align*}
so the scale derivative of~$\psi_\lambda$ is also continuous. Finally, from the continuity of~$\psi_\lambda$ and the fact that $\widehat{\psi}_\lambda$ is an increasing and $\widehat{\phi}_\lambda$ a positive and decreasing function on $\RR_+$ (also note that $\sin (-\pi\nu)>0$), it follows that $\psi_\lambda$ is positive and increasing on $\RR$. After using a similar procedure for~$\phi_\lambda$ we conclude that the functions $\psi_\lambda$ and $\phi_\lambda$ are increasing and decreasing, respectively, positive and continuous solutions of~\eqref{eq_GfLf}.
These functions are also given in the recent paper~\cite{Alili}, albeit there with a different normalization. 

Using the functions $\psi_\lambda$ and $\phi_\lambda$ as defined above, the corresponding Wronskian is given by
\begin{equation*}
w_\lambda = \frac{\diff \psi_\lambda}{\diff S} \phi_\lambda - \frac{\diff \phi_\lambda}{\diff S} \psi_\lambda = \frac{\pi}{\sin (-\pi\nu)},
\end{equation*}
and the Green kernel can be obtained using~\eqref{eq_greenkernel}. In particular we get, for~$y>0$, that
\begin{align}
G_\lambda(0,y) &= w_\lambda^{-1} \psi_\lambda (0)\phi_\lambda (y) \nonumber\\
&= \frac{\sin (-\pi\nu)}{\pi} \widehat{\phi}_\lambda (0) \widehat{\phi}_\lambda (y) \nonumber\\
&= \frac{1}{2 \Gamma(\nu + 1)} \biggl( \frac{\sqrt{2\lambda}}{2} \biggr)^\nu y^{-\nu} K_\nu (y\sqrt{2\lambda}). \label{eq_Ghat}
\end{align}
Note that $G_\lambda(0,y) = \frac{1}{2} \widehat{G}_\lambda(0,y)$, where $\widehat{G}_\lambda$ is the Green kernel for the one-sided Bessel process reflected at zero \cite[p.~137]{handbook} and also that the skewness parameter $\beta$ is not present in the expression for~$G_\lambda(0,y)$. 

A skew two-sided Bessel process has the scaling property; more precisely, \eqref{eq_scaling}~holds with $b=\sqrt{a}$. This can be verified through a straightforward calculation using the Green kernel. 

For a skew two-sided Bessel process, the skewness parameter $\beta$ has a similar interpretation as for a skew Brownian motion, namely that it corresponds to the probability of the process being positive at any given time, when initiated at~0. This follows from
\begin{align}
\int_0^\infty \ee^{-\lambda t} \PP_0(X^{(\nu,\beta)}_t\geq 0) \diff t &= \int_0^\infty G_\lambda(0,y) m_\nu(\diff y) \nonumber\\
&= \int_0^\infty \frac{4\beta}{2 \Gamma(\nu + 1)} \biggl( \frac{\sqrt{2\lambda}}{2} \biggr)^\nu y^{\nu+1} K_\nu (y\sqrt{2\lambda}) \diff y \nonumber\\
&= \frac{\beta}{2^{\nu}\lambda\, \Gamma(1+\nu)} \int_0^\infty z^{1+\nu} K_\nu (z) \diff z \nonumber\\
&= \frac{\beta}{\lambda}, \label{eq_besselG}
\end{align}
where in the last step an integral formula for modified Bessel functions of the second kind \cite[Eq.~(6.561.16)]{Gradshteyn} has been applied. Inverting the Laplace transform gives that, for any $t>0$,
\begin{equation*}
\PP_0(X^{(\nu,\beta)}_t\geq 0) = \beta.
\end{equation*}

\end{example}

\begin{example}\label{ex_skewBM}
\emph{Skew Brownian motion} with skewness parameter $\beta \in (0,1)$ is a diffusion which behaves like a standard Brownian motion when away from a certain skew point, here taken to be 0, while the sign of every excursion from the skew point is chosen by an independent Bernoulli trial with parameter $\beta$, see~\cite{Appuhamillage,Barlow,Lejay,Walsh,Watanabe}. This corresponds to a skew two-sided Bessel process with $\nu=-1/2$, but for the convenience of the reader we write down explicit expressions. 
In this case we have the speed measure 
\begin{equation*}
m(\diff x) = 
\begin{cases}
4(1-\beta) \diff x, & x<0, \\
4\beta \diff x, & x>0,
\end{cases}
\end{equation*}
and the scale function 
\begin{equation*}
S(x) = 
\begin{cases}
\frac{x}{2(1-\beta)}, & x\leq 0, \\
\frac{x}{2\beta}, & x\geq 0.
\end{cases}
\end{equation*}
The fundamental solutions to~\eqref{eq_GfLf} are obtained by inserting $\nu=-1/2$ into the expressions for~$\psi_\lambda$ and $\phi_\lambda$ in the previous section, and after a suitable scaling we get
\begin{equation*}
\psi_\lambda(x) = 
\begin{cases}
\ee^{x\sqrt{2\lambda}} + \frac{1-2\beta}{\beta} \sinh (x\sqrt{2\lambda}), & x\geq 0, \\
\ee^{x\sqrt{2\lambda}}, & x\leq 0,
\end{cases}
\end{equation*}
and 
\begin{equation*}
\phi_\lambda(x) = 
\begin{cases}
\ee^{-x\sqrt{2\lambda}}, & x\geq 0, \\
\ee^{-x\sqrt{2\lambda}} + \frac{1-2\beta}{1-\beta} \sinh (x\sqrt{2\lambda}), & x\leq 0.
\end{cases}
\end{equation*}
The Wronskian is in this case $w_\lambda = 2\sqrt{2\lambda}$. We also get that
\begin{equation}
\label{SBMG}
G_\lambda(0,y) = \frac{1}{2 \sqrt{2\lambda}} \ee^{-y\sqrt{2\lambda}}, \quad y\geq 0.
\end{equation}

\end{example}

\begin{example}\label{ex_osc}
\emph{Oscillating Brownian motion} (see, e.g.,~\cite{Pigato}) is a diffusion $(\widetilde{X}_t)_{t\geq 0}$ which is characterized by the speed measure
\begin{equation*}
\widetilde{m}(\diff x) = 
\begin{cases}
(2/\sigma_{-}^2) \diff x, & x<0, \\
(2/\sigma_{+}^2) \diff x, & x>0,
\end{cases}
\end{equation*}
with $\sigma_{-}>0, \sigma_{+}>0$, and the scale function $\widetilde{S}(x) = x$. The fundamental solutions associated with oscillating Brownian motion are (cf.~\cite{Mordecki}) 
\begin{equation*}
\widetilde{\psi}_\lambda(x) = 
\begin{cases}
\cosh \Big(\frac{x\sqrt{2\lambda}}{\sigma_+}\Big) + \frac{\sigma_+}{\sigma_-} \sinh \Big(\frac{x\sqrt{2\lambda}}{\sigma_+}\Big), & x\geq 0, \\
\exp\Big(\frac{x\sqrt{2\lambda}}{\sigma_-}\Big), & x\leq 0,
\end{cases}
\end{equation*}
and 
\begin{equation*}
\widetilde{\phi}_\lambda(x) = 
\begin{cases}
\exp\Big(-\frac{x\sqrt{2\lambda}}{\sigma_+}\Big), & x\geq 0, \\
\cosh \Big(\frac{x\sqrt{2\lambda}}{\sigma_+}\Big) - \frac{\sigma_+}{\sigma_-} \sinh \Big(\frac{x\sqrt{2\lambda}}{\sigma_+}\Big), & x\leq 0,
\end{cases}
\end{equation*}
and the corresponding Wronskian is given by $\widetilde{w}_\lambda = \sqrt{2\lambda}\big( \frac{1}{\sigma_+} + \frac{1}{\sigma_-} \big)$. This gives that
\begin{equation}
\label{eq_oscG}
\widetilde{G}_\lambda(0,y) = \frac{\sigma_+ \sigma_-}{(\sigma_+ + \sigma_-)\sqrt{2\lambda}} \exp\biggl(-\frac{y\sqrt{2\lambda}}{\sigma_+}\biggr), \quad y\geq 0.
\end{equation}
It can be checked that the oscillating Brownian motion has the scaling property. In fact, it holds that 
\begin{equation}
\label{oscscaling}
 (\widetilde{X}_t)_{t\geq 0} \overset{(d)}{=}(S(X_t))_{t\geq 0},
\end{equation}
where $(X_t)_{t\geq 0}$ denotes the skew Brownian motion with $\beta=\sigma_-/(\sigma_++\sigma_-)$ and $S$ is its scale function. 

\end{example}

\begin{example}\label{ex_sticky}
\emph{Sticky Brownian motion} is a diffusion which behaves like a standard Brownian motion when on excursions away from a certain given point, which we take to be 0. The crucial property which distinguishes sticky Brownian motion from standard Brownian motion is that the occupation time at~0 for sticky Brownian motion is positive up to any given time $t>0$ a.s. if the process starts from~0, i.e.,
$$
\Leb\{ s\in[0,t]\, :\, X_t=0\}>0 \quad \text{$\PP_0\ $-a.s.},
$$
where $(X_t)_{t\geq 0}$ denotes the sticky Brownian motion. However, a sticky Brownian motion does not stay any time interval at~0 (such a behavior would violate the strong Markov property). 
The scale function and the speed measure are 
\begin{align*}
&S(x)=x,\quad m(\diff x) = 2 \diff x +2\gamma \varepsilon_{\{0\}}(dx),
\end{align*}
respectively, where $\gamma>0$ and $\varepsilon_{\{0\}}$ denotes the Dirac measure at~0. The fundamental solutions are (see~\cite[p.~127]{handbook} and references therein)
\begin{equation*}
\psi_\lambda(x) = 
\begin{cases}
\ee^{x\sqrt{2\lambda}} + \gamma\sqrt{2\lambda} \sinh (x\sqrt{2\lambda}), & x\geq 0, \\
\ee^{x\sqrt{2\lambda}}, & x\leq 0,
\end{cases}
\end{equation*}
and 
\begin{equation*}
\phi_\lambda(x) = 
\begin{cases}
\ee^{-x\sqrt{2\lambda}}, & x\geq 0, \\
\ee^{-x\sqrt{2\lambda}} - \gamma\sqrt{2\lambda} \sinh (x\sqrt{2\lambda}), & x\leq 0.
\end{cases}
\end{equation*}
with Wronskian $w_\lambda = 2\sqrt{2\lambda} + 2\lambda\gamma$. In particular, this gives that
\begin{equation}
\label{green0}
G_\lambda(0,y) = \frac{1}{2 \sqrt{2\lambda}+2\lambda\gamma} \ee^{-y\sqrt{2\lambda}}, \quad y\geq 0.
\end{equation}

\end{example}

\section{Kac's moment formula}
\label{section_Kac}

In this section we recall the classical moment formula for integral functionals due to Kac~\cite{Kac}. See~\cite{Fitzsimmons} for formulas for additive functionals in a framework of a general strong Markov process, and also for further references. Our aim here is to present the formula in a form directly applicable to the case at hand. 

Let~$X$ be a regular diffusion taking values on an interval $I\subseteq\RR$, as defined above. For a measurable and bounded function $V$ define for~$t>0$
\begin{equation*}
A_t(V) := \int_0^t V(X_s) \diff s.
\end{equation*}

\begin{proposition}[Kac's moment formula]
For $t>0$, $x\in I$ and $n=1,2,\dotsc$,
\begin{equation}\label{eq_Kac_moment}
\EE_x\bigl((A_t(V))^n\bigr) = n \int_I m(\diff y) \int_0^t p(s;x,y) V(y) \EE_y\bigl((A_{t-s}(V))^{n-1}\bigr) \diff s.
\end{equation}
\end{proposition}

\begin{proof}
The formula clearly holds for~$n=1$. For $n\geq 2$ consider
\begin{align*}
\EE_x\bigl((A_t(V))^n\bigr) &= \EE_x \biggl( \Bigl( \int_0^t V(X_s) \diff s \Bigr)^n \biggr) \\
&= \EE_x \biggl( \int_0^t \dotsi \int_0^t V(X_{s_1}) \cdot\dotsc\cdot V(X_{s_n}) \diff s_1 \cdot\dotsc\cdot \diff s_n \biggr) \\
&= n! \EE_x \biggl( \int_0^t \diff s_1 \int_{s_1}^t \diff s_2 \, \dotsi \int_{s_{n-1}}^t \diff s_n V(X_{s_1}) \cdot\dotsc\cdot V(X_{s_n}) \biggr),
\end{align*}
where the last step holds due to the symmetry of the function
\begin{equation*}
(s_1,\dotsc,s_n) \to V(X_{s_1}) \cdot\dotsc\cdot V(X_{s_n}).
\end{equation*}
Consequently, using the Markov property and the induction assumption,
\begin{align*}
\EE_x\bigl((A_t(V))^n\bigr) &= n! \int_0^t \diff s_1 \int_I m(\diff y) p(s;x,y)V(y) \\
&\qquad \cdot \EE_y \biggl( \int_0^{t-s_1} \diff s_2 
\, \dotsi \int_{s_{n-1}}^{t-s_1} \diff s_n V(X_{s_2}) \cdot\dotsc\cdot V(X_{s_n}) \biggr) \\
&= n \int_0^t \diff s_1 \int_I m(\diff y) p(s;x,y)V(y) \EE_y\bigl((A_{t-s_1}(V))^{n-1}\bigr),
\end{align*} 
which proves the claim.
\hfill\end{proof}

\section{Moment generating function}
\label{MOM}

In this section, as in the previous one, it is assumed that $X$ is a regular diffusion taking values in the interval $I\subseteq\RR$, as introduced in Section~\ref{section_prelim}. Recall that $A_t$ is the occupation time on $\RR_+$ up to time $t$, as defined in~\eqref{def_At}. 
Let $T\sim \operatorname{Exp}(\lambda)$ be an exponentially distributed random variable independent of~$X$. 
We here derive an expression for the moment generating function of~$A_T$, which always exists, since $A_t\leq t$ for all~$t$. 

\begin{theorem}\label{thm_ExerAT}
Let $I^+:=I \cap [0,\infty)$. For $x>0$,
\begin{equation}\label{eq_thmExerAT}
\EE_x(\ee^{-rA_T}) = \frac{\lambda}{\lambda+r} - \left( \frac{\lambda}{\lambda+r} - \EE_0(\ee^{-rA_T}) \right) \EE_x(\ee^{-(\lambda+r)H_0}),
\end{equation}
and
\begin{equation}\label{eq_thmE0}
\EE_0(\ee^{-rA_T}) =\frac{1}{\lambda+r}\left(\lambda+ r\, \frac{\displaystyle 1-\lambda \int_{I^+} G_\lambda (0,y) \,m(\diff y)}{\displaystyle 1+r \int_{I^+} G_\lambda (0,y) \widehat{f}(y;\lambda+r) \,m(\diff y)}\right).
\end{equation}
\end{theorem}

\begin{proof}
Equation~\eqref{eq_thmExerAT} follows from 
\begin{align*}
\EE_x(\ee^{-rA_T} ; H_0>T) &= \EE_x(\ee^{-rT} ; H_0>T) \nonumber\\
&= \frac{\lambda}{\lambda+r} \left( 1 - \EE_x(\ee^{-(\lambda+r)H_0}) \right), 
\end{align*}
together with 
\begin{align*}
\EE_x(\ee^{-rA_T} ; H_0<T) &= \EE_0(\ee^{-rA_T}) \EE_x(\ee^{-rH_0} ; H_0<T) \nonumber\\
&= \EE_0(\ee^{-rA_T})\EE_x(\ee^{-(\lambda+r)H_0}). 
\end{align*}
From Kac's moment formula~\eqref{eq_Kac_moment} with $V(x)=\idop_{[0,\infty)}(x)$ it follows that
\begin{align*}
\EE_x(\ee^{-rA_t}) &= \sum_{n=0}^\infty \EE_x(A_t^n) \frac{(-r)^n}{n!} \\
&= 1 + \sum_{n=1}^\infty (-r) \frac{(-r)^{n-1}}{(n-1)!} \int_{I^+} m(\diff y) \int_0^t p(s;x,y) \EE_y(A_{t-s}^{n-1}) \diff s\\
&= 1 - r \int_{I^+} m(\diff y) \int_0^t p(s;x,y) \EE_y(\ee^{-rA_{t-s}}) \diff s.
\end{align*}
From this we obtain the formula
\begin{align}\label{eq_ExerAT}
\EE_x(\ee^{-rA_T}) &= \int_0^\infty \EE_x(\ee^{-rA_t})\lambda \ee^{-\lambda t} \diff t \nonumber\\
&= 1 - r\lambda \int_{I^+} \LL \left\{ \int_0^t p(s;x,y) \EE_y(\ee^{-rA_{t-s}}) \diff s \right\} m(\diff y) \nonumber\\
&= 1 - r \int_{I^+} G_\lambda(x,y) \LL\{ \lambda \EE_y(\ee^{-rA_t}) \} m(\diff y) \nonumber\\
&= 1 - r \int_{I^+} G_\lambda(x,y) \EE_y(\ee^{-rA_T}) m(\diff y),
\end{align} 
using the properties for the Laplace transform of a convolution. Inserting~\eqref{eq_thmExerAT} into the right hand side of~\eqref{eq_ExerAT} and putting $x=0$ we can solve the resulting expression for~$\EE_0(\ee^{-rA_T})$, which gives
\begin{align*}
\EE_0(\ee^{-rA_T}) &= \frac{1-\frac{r\lambda}{\lambda+r} \int_{I^+} G_\lambda (0,y) (1-\EE_y(\ee^{-(\lambda+r)H_0})) \,m(\diff y)}{1+r \int_{I^+} G_\lambda (0,y) \EE_y(\ee^{-(\lambda+r)H_0}) \,m(\diff y)} \\
&= \frac{\lambda}{\lambda+r}+ \frac{r}{\lambda+r} \cdot \frac{1-\lambda \int_{I^+} G_\lambda (0,y) \,m(\diff y)}{1+r \int_{I^+} G_\lambda (0,y) \widehat{f}(y;\lambda+r) \,m(\diff y)},
\end{align*}
proving the result. 
\hfill\end{proof}

In the next corollary we connect formula~\eqref{eq_thmE0} with the result in~\cite[Cor.~2]{Watanabe}, which is a special case of~\cite[Eq.~(68)]{PY} and also corresponds to~\cite[Eq.~(110)]{Truman}. Here we consider the occupation times on both $\RR_+$ and $\RR_-:=(-\infty,0)$:
\begin{equation*}
A_t^+ := \int_0^t \idop_{[0,\infty)}(X_s)\,\diff s \quad \textrm{and} \quad A_t^- := \int_0^t \idop_{(-\infty,0)}(X_s)\,\diff s,
\end{equation*}
respectively. Formula~\eqref{eq_EApAm} below coincides, when multiplied by $\lambda^{-1}$, with~\cite[Eq.~(68)]{PY} (without the local time term).

\begin{corollary}\label{cor_EApAm}
For $r,q\geq 0$,
\begin{equation}\label{eq_EApAm}
\EE_0(\ee^{-rA_T^{+} - qA_T^{-}}) = \frac{\frac{\lambda}{\lambda+q} \phi_{\lambda+r}(0)\psi_{\lambda+q}^- (0) - \frac{\lambda}{\lambda+r} \psi_{\lambda+q}(0)\phi_{\lambda+r}^- (0)}{\phi_{\lambda+r}(0)\psi_{\lambda+q}^- (0) - \psi_{\lambda+q}(0)\phi_{\lambda+r}^- (0)},
\end{equation}
where the superscript~$^-$ denotes the left derivative with respect to the scale function~$S$.
\end{corollary}

\begin{remark}
If the point 0 is included in $A_t^-$ instead, rather than in $A_t^+$, then the left derivatives in~\eqref{eq_EApAm} should be replaced by right derivatives. Note, however, that there is a difference only if~0 is a sticky point, since otherwise the left and right scale derivatives are equal. 
\end{remark}

\begin{proof}[Proof of Corollary~\ref{cor_EApAm}]
We first prove~\eqref{eq_EApAm} when $q=0$. Since for~$y>0$ we have that
\begin{equation*}
\widehat{f}(y;\lambda+r) = \EE_y(\ee^{-(\lambda+r)H_0}) = \frac{\phi_{\lambda+r}(y)}{\phi_{\lambda+r}(0)}, 
\end{equation*}
we can rewrite~\eqref{eq_thmE0} as
\begin{align}
\EE_0(\ee^{-rA_T}) &= \frac{1}{\lambda+r} \left( \lambda + r\, \frac{w_\lambda \phi_{\lambda+r}(0)-\psi_\lambda(0) \phi_{\lambda+r}(0) \int_{I^+} \lambda\phi_\lambda(y) \,m(\diff y)}{w_\lambda \phi_{\lambda+r}(0)+r\psi_\lambda(0) \int_{I^+} \phi_\lambda(y) \phi_{\lambda+r}(y) \,m(\diff y)} \right). \label{eq_thmE2}
\end{align}
For the integral in the numerator it holds that
\begin{equation}
\label{ekv10}
\intco{a}{b} \lambda\phi_\lambda(y) \,m(\diff y) = \intco{a}{b} \frac{\diff}{\diff m}\frac{\diff}{\diff S}\phi_\lambda(y) \,m(\diff y) = \phi_\lambda^- (b)-\phi_\lambda^- (a). 
\end{equation}
Recall the following integration by parts formula for a Lebesgue--Stieltjes integral, with functions $U$ and $V$ being of finite variation and at least one of them continuous on $(a,b)$:
\begin{equation}
\label{ekv11}
\intco{a}{b} U \diff V + \intco{a}{b} V \diff U = U(b-)V(b-) - U(a-)V(a-).
\end{equation}
Since $\phi$ is continuous and of finite variation, we get, applying~\eqref{ekv10} and~\eqref{ekv11}, that
\begin{align}
&\lambda \intco{a}{b} \phi_\lambda(y) \phi_{\lambda+r}(y) \,m(\diff y) 
= \intco{a}{b} \phi_{\lambda+r}(y) \cdot \lambda\phi_{\lambda}(y) \,m(\diff y) \nonumber\\
&\quad= \phi_{\lambda+r}(b) \phi_{\lambda}^- (b) - \phi_{\lambda+r}(a) \phi_{\lambda}^- (a) -\intco{a}{b}\left( \frac{\diff}{\diff S}\phi_{\lambda}(y) \right) \left( \frac{\diff}{\diff S}\phi_{\lambda+r}(y) \right) \diff S, \label{eq_I_lambda}
\end{align}
and likewise
\begin{align}
&(\lambda+r) \intco{a}{b} \phi_\lambda(y) \phi_{\lambda+r}(y) \,m(\diff y) 
= \intco{a}{b} \phi_{\lambda}(y) \cdot (\lambda+r)\phi_{\lambda+r}(y) \,m(\diff y) \nonumber\\
&\quad= \phi_{\lambda}(b) \phi_{\lambda+r}^- (b) -\phi_{\lambda}(a) \phi_{\lambda+r}^- (a) - \intco{a}{b}\left( \frac{\diff}{\diff S}\phi_{\lambda}(y) \right) \left( \frac{\diff}{\diff S}\phi_{\lambda+r}(y) \right) \diff S. \label{eq_I_lambda_r}
\end{align}
Subtracting~\eqref{eq_I_lambda} from~\eqref{eq_I_lambda_r} yields
\begin{multline*}
r \intco{a}{b} \phi_\lambda(y) \phi_{\lambda+r}(y) \,m(\diff y) = \phi_{\lambda}(b) \phi_{\lambda+r}^- (b) - \phi_{\lambda+r}(b) \phi_{\lambda}^- (b) \\ 
+ \phi_{\lambda+r}(a) \phi_{\lambda}^- (a) - \phi_{\lambda}(a) \phi_{\lambda+r}^- (a+).
\end{multline*}
Furthermore, if $b$ is an upper boundary point of the diffusion which is either non-exit or exit-and-entrance with reflection, then it holds that $\phi_{\lambda}^- (b) = 0$ (see~\cite[p.~130, Table~1]{Ito}).
Since the lower boundary point of~$I^+$ is 0 we thereby get that
\begin{align*}
\int_{I^+} \lambda\phi_\lambda(y) \,m(\diff y) &= -\phi_\lambda^- (0), \\
r \int_{I^+} \phi_\lambda(y) \phi_{\lambda+r}(y) \,m(\diff y) &= \phi_{\lambda+r}(0) \phi_{\lambda}^- (0) - \phi_{\lambda}(0) \phi_{\lambda+r}^- (0).
\end{align*}
Inserting this into~\eqref{eq_thmE2} yields
\begin{align}\label{eq_corE0}
\EE_0(\ee^{-rA_T^{+}}) &= \frac{\lambda}{\lambda+r} + \frac{\frac{r}{\lambda+r} ( w_\lambda \phi_{\lambda+r}(0) + \psi_\lambda(0) \phi_{\lambda+r}(0) \phi_\lambda^- (0) )}{w_\lambda \phi_{\lambda+r}(0) + \psi_\lambda(0) ( \phi_{\lambda+r}(0) \phi_{\lambda}^- (0) - \phi_{\lambda}(0) \phi_{\lambda+r}^- (0))} \nonumber\\
&= \frac{\phi_{\lambda+r}(0)\psi_{\lambda}^- (0) - \frac{\lambda}{\lambda+r} \psi_{\lambda}(0)\phi_{\lambda+r}^- (0)}{\phi_{\lambda+r}(0)\psi_{\lambda}^- (0) - \psi_{\lambda}(0)\phi_{\lambda+r}^- (0)},
\end{align}
where $w_\lambda = \psi^{-}_{\lambda}(0)\phi_\lambda (0)-\psi_\lambda (0)\phi^{-}_{\lambda}(0)$ has been inserted, see~\eqref{eq_wronskian}.
We now extend this result. Since $A_T^{+}+A_T^{-}=T$, 
\begin{equation*}
\EE_0(\ee^{-pA_T^{+}-qT}) = \EE_0(\ee^{-(p+q)A_T^{+}-qA_T^{-}}) = \EE_0(\ee^{-rA_T^{+}-qA_T^{-}}),
\end{equation*}
where $r=p+q$. On the other hand, 
\begin{align*}
\EE_0(\ee^{-pA_T^{+}-qT}) &= \int_0^\infty \lambda \ee^{-\lambda t} \EE_0(\ee^{-pA_T^{+}-qt}) \diff t \\
&= \int_0^\infty \lambda \ee^{-(\lambda+q)t} \EE_0(\ee^{-pA_T^{+}}) \diff t \\
&= \frac{\lambda}{\lambda+q} \EE_0(\ee^{-pA_{\widehat{T}}^{+}}), 
\end{align*}
where $\widehat{T}\sim \operatorname{Exp}(\lambda+q)$. Applying~\eqref{eq_corE0} gives
\begin{align*}
\EE_0(\ee^{-rA_T^{+} - qA_T^{-}}) &= \frac{\lambda}{\lambda+q} \EE_0(\ee^{-(r-q) A_{\widehat{T}}^{+}}) \\
&= \frac{\lambda}{\lambda+q} \cdot \frac{\phi_{\lambda+r}(0)\psi_{\lambda+q}^- (0) - \frac{\lambda+q}{\lambda+r} \psi_{\lambda+q}(0)\phi_{\lambda+r}^- (0)}{\phi_{\lambda+r}(0)\psi_{\lambda+q}^- (0) - \psi_{\lambda+q}(0)\phi_{\lambda+r}^- (0)} \\
&= \frac{\frac{\lambda}{\lambda+q}\phi_{\lambda+r}(0)\psi_{\lambda+q}^- (0) - \frac{\lambda}{\lambda+r} \psi_{\lambda+q}(0)\phi_{\lambda+r}^- (0)}{\phi_{\lambda+r}(0)\psi_{\lambda+q}^- (0) - \psi_{\lambda+q}(0)\phi_{\lambda+r}^- (0)},
\end{align*}
which proves~\eqref{eq_EApAm}. 
\hfill\end{proof}

\begin{remark}
If we instead consider the occupation times on the intervals $[\alpha,+\infty)$ and $(-\infty,\alpha)$ for some $\alpha\in\RR$, then~\eqref{eq_EApAm} holds when~0 is replaced with $\alpha$ everywhere. 
\end{remark}

\section{Recursive formula for the moments}\label{section_moments}

In this section we use Kac's moment formula to derive our main result, namely a recursive equation for the Laplace transforms of the moments of~$A_t$ for fixed $t>0$. When the diffusion is a self-similar process, the expression becomes a recursion for moments of~$A_1$ instead (which is easily transformed into a recursion for moments of~$A_t$, since in this case $A_t=tA_1$). It is assumed that $X$ is a regular diffusion taking values in the interval $I\subseteq\RR$ as defined in Section~\ref{section_prelim}. We introduce the Laplace transform of~$A_t$ via
\begin{equation*}
\widehat{A}_x(\lambda;n) := \LL_t\{\EE_x(A_t^n)\}(\lambda) := \int_0^\infty \ee^{-\lambda t} \EE_x(A_t^n) \diff t.
\end{equation*}
If there is no ambiguity, the variables $t$ and $\lambda$ in the notation of the Laplace transforms are omitted; for instance, we shall write $\LL\{\EE_x(A_t^n)\}$ instead of $\LL_t\{\EE_x(A_t^n)\}(\lambda)$. 

\begin{theorem}\label{thm_rec}
Let $I^+:=I \cap [0,\infty)$. The Laplace transforms of the moments of~$A_t$ for~$X$ starting from~0 are given for~$n=1$ by
\begin{equation}\label{eq_n1}
\widehat{A}_0(\lambda;1)= \frac{1}{\lambda} \int_{I^+} G_\lambda(0,y) m(\diff y),
\end{equation}
and for~$n=2,3,\dotsc$ by 
\begin{equation}\label{eq_Ahat_rec}
\widehat{A}_0(\lambda;n) = \frac{n!}{\lambda^{n-1}} \widehat{A}_0(\lambda;1) + \frac{n!}{\lambda^{n+1}}\sum_{k=1}^{n-1} \left( 1 - \frac{\lambda^{n-k+1}}{(n-k)!} \widehat{A}_0(\lambda;n-k) \right) D_{k}(\lambda), 
\end{equation}
where
\begin{equation}\label{eq_Ck}
D_k(\lambda) := \frac{(-\lambda)^k}{(k-1)!} \int_{I^+} G_\lambda(0,y) \widehat{f}_\lambda^{(k-1)}(y;\lambda) m(\diff y).
\end{equation}
Moreover, under the scaling property~\eqref{eq_scaling}, for all~$\lambda>0$,
\begin{equation}\label{eq_E0A1}
\EE_0(A_1) = \lambda\int_{I^+} G_\lambda(0,y) m(\diff y),
\end{equation}
and
\begin{equation}\label{eq_E0A1n}
\EE_0(A_1^n) = \EE_0(A_1) + \sum_{k=1}^{n-1} \left(1 - \EE_0(A_1^{n-k}) \right) D_k(\lambda).
\end{equation}
\end{theorem}

\begin{remark}\label{rem_U}
Equation~\eqref{eq_Ahat_rec} can be rewritten as
\begin{equation*}
U_n(\lambda) = U_1(\lambda) + \sum_{k=1}^{n-1} \left(1 - U_{n-k}(\lambda)) \right) D_k(\lambda), 
\end{equation*}
where $U_n(\lambda):=\frac{\lambda^{n+1}}{n!} \widehat{A}_0(\lambda;n)$. 
\end{remark}

\begin{proof}[Proof of Theorem~\ref{thm_rec}]
Taking the Laplace transform on both sides of Kac's moment formula~\eqref{eq_Kac_moment} with $V(x)=\idop_{[0,\infty)}(x)$ yields 
\begin{equation} \label{eq_Kac}
\widehat{A}_x(\lambda;n) = n \int_{I^+} G_\lambda(x,y) \widehat{A}_y(\lambda;n-1) m(\diff y).
\end{equation}
Substituting $n=1$ and $x=0$ in~\eqref{eq_Kac} gives~\eqref{eq_n1}. To derive~\eqref{eq_Ahat_rec} we first find an expression for~$\widehat{A}_x(\lambda;n)$ in terms of $\widehat{A}_0(\lambda;k)$ for different~$k$. 
For any starting point $x>0$,
\begin{equation*}
A_t = 
\begin{cases}
H_0 + A_{t-H_0} \circ \theta_{H_0}, &H_0<t,\\
t, &H_0>t.
\end{cases}
\end{equation*}
where $\theta_t$ is the usual shift operator. By the strong Markov property, 
\begin{equation}\label{eq11}
\EE_x(A_t^n) = t^n(1-\PP_x(H_0<t)) + \sum_{k=0}^n \binom{n}{k} \EE_x(H_0^k (A_{t-H_0} \circ \theta_{H_0})^{n-k};H_0<t).
\end{equation}
We have the following Laplace transforms with respect to~$t$:
\begin{align*}
\LL\{\EE_x(H_0^k (A_{t-H_0} \circ \theta_{H_0})^{n-k};H_0<t)\} &= \LL\left\{ \int_0^t \EE_0(A_{t-s}^{n-k})s^k f(x;s) \diff s \right\} \\
&= \LL\{\EE_0(A_{t}^{n-k})\} \LL\{t^k f(x;t)\} \\
&= (-1)^k \widehat{A}_0(\lambda;n-k) \widehat{f}_\lambda^{(k)}(x;\lambda)
\end{align*}
and
\begin{align*}
\LL\{t^n \PP_x(H_0<t)\} &= \LL\left\{ t^n \int_0^t f(x;t) \diff t\right\} \\
&= (-1)^n \frac{\diff^n}{\diff\lambda^n} \left(\frac{1}{\lambda}\widehat{f}(x;\lambda)\right) \\
&= \frac{n!}{\lambda^{n+1}} \sum_{k=0}^n \frac{(-\lambda)^k}{k!} \widehat{f}_\lambda^{(k)}(x;\lambda).
\end{align*}
Hence, taking the Laplace transforms on the both sides of~\eqref{eq11} gives
\begin{align}
\widehat{A}_x(\lambda;n) &= \frac{n!}{\lambda^{n+1}} - \frac{n!}{\lambda^{n+1}} \sum_{k=0}^n \frac{(-\lambda)^k}{k!} \widehat{f}_\lambda^{(k)}(x;\lambda) \nonumber\\
&\hspace{9em} + \sum_{k=0}^n (-1)^k\binom{n}{k} \widehat{A}_0(\lambda;n-k) \widehat{f}_\lambda^{(k)}(x;\lambda) \nonumber\\
&= \frac{n!}{\lambda^{n+1}} + \frac{n!}{\lambda^{n+1}} \sum_{k=0}^{n-1} \frac{(-\lambda)^{k}}{k!} \widehat{f}_\lambda^{(k)}(x;\lambda) \left( \frac{\lambda^{n-k+1}}{(n-k)!}\widehat{A}_0(\lambda;n-k) - 1 \right). \label{eq_LXtn}
\end{align}
Note that in the summation the term with $k=n$ disappears, since $\widehat{A}_0(\lambda;0) = 1/\lambda$. 
Inserting the expression in~\eqref{eq_LXtn} into both sides of~\eqref{eq_Kac} and solving for~$\widehat{A}_0(\lambda;n)$ gives 
\begin{multline}\label{eq_Ahat_rec2}
\widehat{A}_0(\lambda;n) = \frac{n!}{\lambda^{n+1} \widehat{f}(x;\lambda)} \biggl( \lambda\int_{I^+} G_\lambda(x,y) m(\diff y) + \widehat{f}(x;\lambda) - 1 \\
+ \sum_{k=1}^{n-1} \left( 1 - \frac{\lambda^{n-k+1}}{(n-k)!}\widehat{A}_0(\lambda;n-k) \right) D_k(x;\lambda) \biggr),
\end{multline}
where
\begin{equation}\label{DKL}
D_k(x;\lambda) = \frac{(-\lambda)^k}{k! \widehat{f}(x;\lambda)} \left( \widehat{f}_\lambda^{(k)}(x;\lambda) + k\int_{I^+} G_\lambda(x,y) \widehat{f}_\lambda^{(k-1)}(y;\lambda) m(\diff y) \right).
\end{equation}
Note, however, that $x$ is not present on the left hand side of~\eqref{eq_Ahat_rec2}, so the right hand side cannot depend on $x$ either. Thus we may choose any $x>0$. We show now that the limit of the right hand side of~\eqref{eq_Ahat_rec2} exists when $x\downarrow 0$. Since
\begin{align*}
&\lim_{x\downarrow 0} \int_{I^+} G_\lambda(x,y) m(\diff y) \\
&\qquad= \lim_{x\downarrow 0} \biggl( \frac{\phi_\lambda(x)}{w_\lambda} \intco{0}{x} \psi_\lambda(y) m(\diff y) + \frac{\psi_\lambda(x)}{w_\lambda} \int\limits_{[x,\infty)\cap I^+} \hspace{-0.7em}\phi_\lambda(y) m(\diff y) \biggr) \\
&\qquad= \int_{I^+} G_\lambda(0,y) m(\diff y) \\
&\qquad= \lambda \widehat{A}_0(\lambda;1), 
\end{align*}
it is seen by induction that $\lim_{x\downarrow 0} D_{k}(x;\lambda)$ exists for all values of~$k.$ Consequently, recalling that $\lim_{x\downarrow 0}\widehat{f}(x;\lambda)=1$, we may write 
\begin{equation*}
\widehat{A}_0(\lambda;n) = \frac{n!}{\lambda^{n-1}} \widehat{A}_0(\lambda;1) + \frac{n!}{\lambda^{n+1}}\sum_{k=1}^{n-1} \left( 1 - \frac{\lambda^{n-k+1}}{(n-k)!} \widehat{A}_0(\lambda;n-k) \right) \lim_{x\downarrow 0} D_{k}(x;\lambda). 
\end{equation*}
We calculate the limit of~$D_{k}(x;\lambda)$ from the explicit expression~\eqref{DKL}. Using the result in Lemma~\ref{lemma_f0}, 
\begin{align*}
\lim_{x\downarrow 0} D_k(x;\lambda) &= \lim_{x\downarrow 0} \frac{(-\lambda)^k}{(k-1)!} \int_{I^+} G_\lambda(x,y) \widehat{f}_\lambda^{(k-1)}(y;\lambda) m(\diff y) \\
&= \frac{(-\lambda)^k}{(k-1)!} \lim_{x\downarrow 0} \biggl( \frac{\phi_\lambda(x)}{w_\lambda} \intco{0}{x} \psi_\lambda(y) \widehat{f}_\lambda^{(k-1)}(y;\lambda) m(\diff y) \\
& \hspace{8em} + \frac{\psi_\lambda(x)}{w_\lambda} \int\limits_{[x,\infty)\cap I^+} \hspace{-0.5em} \phi_\lambda(y) \widehat{f}_\lambda^{(k-1)}(y;\lambda) m(\diff y) \biggr) \\
&=\frac{(-\lambda)^k}{(k-1)!} \int_{I^+} G_\lambda(0,y) \widehat{f}_\lambda^{(k-1)}(y;\lambda) m(\diff y), 
\end{align*}
which is the right hand side of~\eqref{eq_Ck}. The proof of the recursive equation~\eqref{eq_Ahat_rec} is now complete.

Assume finally that the process~$X$ is self-similar so that~\eqref{scalingA} holds. Then $\EE_0(A_t^n) = t^n \EE_0(A_1^n)$ and thus 
\begin{equation}\label{eq_scAE}
\widehat{A}_0(\lambda;n) = \LL\{t^n\EE_0(A_1^n)\} = \frac{n!}{\lambda^{n+1}} \EE_0(A_1^n). 
\end{equation}
It is clear that in this case~\eqref{eq_E0A1} follows immediately from~\eqref{eq_n1}. 
After inserting~\eqref{eq_scAE} into~\eqref{eq_Ahat_rec} it is seen that $\lambda$ is only left in $D_k(\lambda),\, k=1,2,...,n-1$. Putting $n=2$ we find that $D_1(\lambda)$ does not depend on $\lambda$, and, by induction, we conclude that $D_k(\lambda)$ does not depend on $\lambda$ for any $k$.
\hfill\end{proof}

\section{Examples}\label{section_examples}
\subsection{Skew two-sided Bessel processes}

We now apply the results in Theorems~\ref{thm_ExerAT} and~\ref{thm_rec} on the skew two-sided Bessel process, which is described in Example~\ref{ex_bessel}. In this case the function $\widehat{f}$ is
\begin{equation}\label{eq_fbes}
\widehat{f}(x;\lambda) = \frac{\widehat{\phi}(x)}{\widehat{\phi}(0)} = \frac{2^{\nu+1}}{\Gamma(-\nu)} (x\sqrt{2\lambda})^{-\nu} K_\nu (x\sqrt{2\lambda}), 
\end{equation}
and the Green kernel is given in~\eqref{eq_Ghat}. Note that here $I^+ = [0,\infty)$. We derive the next result from~\eqref{eq_thmE0} in Theorem~\ref{thm_ExerAT}. Alternatively, formula~\eqref{eq_EApAm} in Corollary~\ref{cor_EApAm} could have been used. 

\begin{proposition}
Let $T\sim \textrm{Exp}(\lambda)$ independent of~$X$. For any $r\geq 0$,
\begin{equation}\label{eq_LTbessel}
\EE_0(\ee^{-rA_T}) = \frac{\beta \lambda^{\nu+1} + (1-\beta) (\lambda+r)^{\nu+1}}{\beta (\lambda+r) \lambda^{\nu} + (1-\beta) (\lambda+r)^{\nu+1}}.
\end{equation}
\end{proposition}

\begin{proof}
The identity is trivial when $r=0$. For $r>0$ it follows from~\eqref{eq_thmE0} that
\begin{align}\label{eq_deltas}
\EE_0(\ee^{-rA_T}) &= \frac{1}{\lambda+r} \left( \lambda + r \, \frac{1-\lambda \int_0^\infty G_\lambda (0,y) \,m_\nu(\diff y)}{1+r \int_0^\infty G_\lambda (0,y) \widehat{f}(y;\lambda+r) \,m_\nu(\diff y)} \right) \nonumber\\
&= \frac{1}{\lambda+r} \left( \lambda + r \, \frac{1-\lambda \Delta_1}{1+r \Delta_2} \right),
\end{align}
where we need to calculate the integrals $\Delta_1$ and $\Delta_2$. From~\eqref{eq_besselG} we already have that
\begin{equation*}
\Delta_1 = \int_0^\infty G_\lambda(0,y) m_\nu(\diff y) = \frac{\beta}{\lambda}.
\end{equation*}
Recalling~\eqref{eq_Ghat} and~\eqref{eq_fbes}, the second integral becomes
\begin{align}\label{eq_delta2}
\Delta_2 &= \int_0^\infty G_\lambda(0,y) \widehat{f}(y;\lambda+r) \, m_\nu(\diff y) \nonumber\\
&= \frac{4\beta}{\Gamma(\nu+1)\Gamma(-\nu)} \left( \frac{\lambda}{\lambda+r} \right)^{\frac{\nu}{2}} \int_0^\infty y K_\nu \big(y\sqrt{2(\lambda+r)}\big) K_\nu \big(y\sqrt{2\lambda}\big) \diff y \nonumber\\
&= \frac{-\beta\nu}{\lambda+r} \left( \frac{\lambda}{\lambda+r} \right)^{\nu} {}_2F_1 \left( 1,1+\nu\,;2\,;\frac{r}{\lambda+r}\right),
\end{align}
after applying an integral formula for modified Bessel functions of the second kind~\cite[Eq.~(6.576.4)]{Gradshteyn}. The hypergeometric function can in this case be rewritten using an incomplete beta function \cite[Eq.~(8.391)]{Gradshteyn} as 
\begin{align*}
{}_2F_1 \left( 1,1+\nu\,;2\,;\frac{r}{\lambda+r}\right) &= \frac{\lambda+r}{r} \int_0^\frac{r}{\lambda+r} (1-t)^{-\nu-1} \diff t \\
&= \frac{\lambda+r}{r\nu} \left( \left( \frac{\lambda}{\lambda+r} \right)^{-\nu} -1 \right).
\end{align*}
Inserting this into~\eqref{eq_delta2} yields
\begin{equation*}
\Delta_2 = \frac{\beta}{r} \left( \left( \frac{\lambda}{\lambda+r} \right)^{\nu} -1 \right). 
\end{equation*}
When inserting $\Delta_1$ and $\Delta_2$ into~\eqref{eq_deltas} we obtain
\begin{align*}
\EE_0(\ee^{-rA_T}) &= \frac{1}{\lambda+r} \left( \lambda + r \, \frac{1-\beta}{1+\beta \Big( \big( \frac{\lambda}{\lambda+r} \big)^{\nu} -1 \Big)} \right) \\
&= \frac{1}{\lambda+r} \left( \lambda + \frac{r(1-\beta)(\lambda+r)^\nu}{(1-\beta)(\lambda+r)^\nu+\beta \lambda^\nu} \right) \\
&= \frac{(1-\beta)\lambda(\lambda+r)^\nu + \beta \lambda^{\nu+1} + (1-\beta)r(\lambda+r)^\nu}{(\lambda+r)((1-\beta)(\lambda+r)^\nu + \beta \lambda^\nu)} \\
&= \frac{\beta \lambda^{\nu+1} + (1-\beta) (\lambda+r)^{\nu+1}}{\beta (\lambda+r) \lambda^{\nu} + (1-\beta) (\lambda+r)^{\nu+1}},
\end{align*}
which proves the claim. 
\hfill\end{proof}

Note that the expression in~\eqref{eq_LTbessel} can be rewritten as
\begin{equation*}
\lambda \int_0^\infty \ee^{-\lambda t} \EE(\ee^{-r A_t}) \diff t = \EE_0(\ee^{-rA_T}) = \lambda \cdot \frac{\beta (\lambda + r)^{-\nu-1} + (1-\beta) \lambda^{-\nu-1}}{\beta (\lambda + r)^{-\nu} + (1-\beta) \lambda^{-\nu}},
\end{equation*}
which is equivalent to (4.a) in~\cite{Barlow}; see also \cite{Lamperti,Watanabe}. 

Next we apply Theorem~\ref{thm_rec} to find a recursive formula for the moments of~$A_1$. 

\begin{theorem}[Skew two-sided Bessel process]\label{thm_besselrec}
For $n\geq 1$,
\begin{equation}\label{eq_besselrec}
\EE_0(A_1^n) = \beta \binom{\nu+n-1}{n-1} - \beta \sum_{k=1}^{n-1} \binom{\nu+k-1}{k} \EE_0(A_1^{n-k}).
\end{equation}
\end{theorem}

\begin{proof}
Equations~\eqref{eq_Ck}--\eqref{eq_E0A1n} hold, since the skew two-sided Bessel process is self-similar. From~\eqref{eq_E0A1} and~\eqref{eq_besselG} we obtain the first moment, 
\begin{equation}\label{eq_1stmom}
\EE_0(A_1) = \lambda\int_0^\infty G_\lambda(0,y) m_\nu(\diff y) = \beta.
\end{equation}
Next we calculate the coefficients $D_k(\lambda)$ as given in~\eqref{eq_Ck}. Setting $k=1$ gives
\begin{align*}
D_{1}(\lambda) &= -\lambda \int_0^\infty G_\lambda(0,y) \widehat{f}(y;\lambda) m_\nu(\diff y) \\
&= \beta\nu,
\end{align*}
where the integral is given by~\eqref{eq_delta2} with $r=0$ (note that the hypergeometric function is equal to~1 in this case). 
In order to find $D_{k}(\lambda)$ for~$k>1$ we need to differentiate $\widehat{f}(x;\lambda)$ with respect to~$\lambda$. Let 
\begin{equation*}
g(\nu; x) := x^{-\nu} K_\nu (x),
\end{equation*}
for which it can be shown by induction that
\begin{align*}
g^{(n)}(\nu;x) := \frac{\diff^n}{\diff x^n}g(\nu, x) &= \sum_{i=0}^{\lfloor n/2 \rfloor} \frac{(-1)^{n+i} n!}{i! (n-2i)! 2^{i}}x^{-(\nu+i)} K_{\nu+n-i} (x) \\
&= \sum_{i=\lceil n/2 \rceil}^{n} \frac{(-1)^{i} n!}{(n-i)! (2i-n)! 2^{n-i}}x^{-(\nu+n-i)} K_{\nu+i} (x). 
\end{align*}
Writing 
\begin{equation*}
\widehat{f}(x;\lambda) = \frac{2^{\nu+1}}{\Gamma(-\nu)} g(\nu,x\sqrt{2\lambda}),
\end{equation*}
and differentiating (see, e.g.,~\cite[Appx.~5]{handbook} for general formulas) gives 
\begin{align*}
\widehat{f}_\lambda^{(k)}(x;\lambda) &= \frac{2^{\nu+1}}{\Gamma(-\nu)} \sum_{i=1}^{k} \frac{(-1)^{k-i} (2k-1-i)!}{(i-1)! (k-i)!} \frac{(x\sqrt{2})^{i}}{(2\sqrt{\lambda})^{2k-i}} g^{(i)}(\nu;x\sqrt{2\lambda}) \\
 &= \frac{2^{\nu+1}}{\lambda^{k}\Gamma(-\nu)} \sum_{i=1}^{k} \sum_{j=\lceil i/2 \rceil}^{i} \frac{(-1)^{k+i-j} (2k-1-i)! \,i}{(k-i)! (i-j)! (2j-i)! 2^{2k-j}} \\
 &\hspace{14em} \cdot (x\sqrt{2\lambda})^{-(\nu-j)} K_{\nu+j} (x\sqrt{2\lambda}).
\end{align*}
Combining this with~\eqref{eq_Ghat} yields, for~$k=1,2,\dotsc$,
\begin{align*}
D_{k+1}(\lambda) &= \frac{(-\lambda)^{k+1}}{k!} \int_0^\infty G_\lambda(0,y) \widehat{f}_\lambda^{(k)}(y;\lambda) m_\nu(\diff y) \\
&=\frac{-2\beta}{k! \,\Gamma(\nu+1)\Gamma(-\nu)} \sum_{i=1}^{k} \sum_{j=\lceil i/2 \rceil}^{i} \frac{(-1)^{i-j} (2k-1-i)! \,i}{(k-i)! (i-j)! (2j-i)! 2^{2k-j}} \cdot \Delta_{\nu,j},
\end{align*}
where 
\begin{align*}
\Delta_{\nu,j} &:= \sqrt{2\lambda} \int_0^\infty (y\sqrt{2\lambda})^{j+1} K_\nu (y\sqrt{2\lambda}) K_{\nu+j} (y\sqrt{2\lambda}) \diff y \\
&= \int_0^\infty z^{j+1} K_\nu (z) K_{\nu+j}(z) \diff z \\
&= \frac{2^{j-1}}{j+1}\Gamma(1-\nu)\Gamma(\nu+1+j), 
\end{align*}
by an integration formula for modified Bessel functions \cite[Eq.~(6.576.4)]{Gradshteyn}. Inserting this and changing the order of summation, 
\begin{align*}
D_{k+1}(\lambda) &=\frac{-\beta\,\Gamma(1-\nu)}{k! \,\Gamma(\nu+1)\Gamma(-\nu)} \sum_{i=1}^{k} \sum_{j=\lceil i/2 \rceil}^{i} \frac{(-1)^{i-j} (2k-1-i)! \,i\, \Gamma(\nu+1+j)}{(k-i)! (i-j)! (2j-i)! (j+1) 2^{2k-2j}} \\
&=\frac{\beta\nu}{k! \,\Gamma(\nu+1)} \sum_{j=1}^{k} \frac{\Gamma(\nu+1+j)}{(j+1) 2^{2k-2j}} \sum_{i=j}^{\min (k,2j)} \frac{(-1)^{i-j} (2k-1-i)!\,i }{(k-i)! (i-j)! (2j-i)!} \\
&=\frac{\beta}{k \,\Gamma(\nu)} \sum_{j=1}^{k} \frac{\Gamma(\nu+1+j)}{(j+1)! \, 2^{2k-2j}} \sum_{i=j}^{\min (k,2j)} i (-1)^{i-j} \binom{j}{i-j}\binom{2k-1-i}{k-1}.
\end{align*}
In the case $j=k$, the inner sum has only one term, namely
\begin{equation*}
\sum_{i=k}^{k} i (-1)^{i-k} \binom{k}{i-k} \binom{2k-1-i}{k-1} = k. 
\end{equation*}
When $j<k$, we show that the inner sum is zero. Note that for any $k<i<2k$ the summand is zero, since $0 \leq 2k-1-i < k-1$. Thus, when $j<k$ we can always choose~$2j$ as the upper limit for the summation index, and the inner sum becomes
\begin{align*}
\sum_{i=j}^{2j} i (-1)^{i-j} \binom{j}{i-j} \binom{2k-1-i}{k-1} &= \sum_{i=0}^{j} (i+j) (-1)^{i} \binom{j}{i} \binom{2k-1-j-i}{k-1}\\
&= j \Biggl[ \sum_{i=0}^{j} (-1)^{i} \binom{j}{i} \binom{2k-1-j-i}{k-1} \\
&\qquad+ \sum_{i=1}^{j} (-1)^{i} \binom{j-1}{i-1} \binom{2k-1-j-i}{k-1} \Biggr] \\
&= j \left( \binom{2k-2j-1}{k-j-1} - \binom{2k-2j-1}{k-j} \right) \\
&= 0,
\end{align*}
using, in the third step, the identity \cite[Eq.~(3.49)]{Gould}
\begin{equation*}
\sum_{k=0}^{n} (-1)^{k} \binom{n}{k}\binom{x-k}{m} = \binom{x-n}{m-n},
\end{equation*}
valid for $n,m\in\NN$ and $x\notin\{m,\dotsc,n-1\}$. From this we conclude that only the term corresponding to~$j=k$ remains in the expression for~$D_{k+1}(\lambda)$, which thus simplifies to
\begin{equation*}
D_{k+1}(\lambda) = \frac{\beta\,\Gamma(\nu+1+k)}{(k+1)! \,\Gamma(\nu)} = \beta \binom{\nu+k}{k+1}. 
\end{equation*}
Reducing the index $k$ by 1 and recalling that $D_1(\lambda)=\beta\nu$, we conclude that
\begin{equation}\label{eq_besselD}
D_{k}(\lambda) = \beta \binom{\nu+k-1}{k},
\end{equation}
for all~$k\geq 1$. Inserting this and~\eqref{eq_1stmom} into~\eqref{eq_E0A1n} results in the recursion
\begin{align*}
\EE_0(A_1^n) &= \beta + \beta \sum_{k=1}^{n-1} \left(1 - \EE_0(A_1^{n-k}) \right) \binom{\nu+k-1}{k} \\
&= \beta \sum_{k=0}^{n-1} \binom{\nu+k-1}{k} - \beta \sum_{k=1}^{n-1} \binom{\nu+k-1}{k} \EE_0(A_1^{n-k}) \\
&= \beta \binom{\nu+n-1}{n-1} - \beta \sum_{k=1}^{n-1} \binom{\nu+k-1}{k} \EE_0(A_1^{n-k}),
\end{align*}
where the last step follows by the identity \cite[Eq.~(1.49)]{Gould} 
\begin{equation}\label{eq_Gsum}
\sum_{k=0}^n \binom{x+k}{k} = \binom{x+n+1}{n}.
\end{equation}
This proves the theorem. 
\hfill\end{proof}

\begin{corollary}
The mapping $\beta \mapsto \EE_0(A_1^n)$ is continuous and increasing. 
\end{corollary}

\begin{proof}
Since $\nu\in(-1,0)$ it follows from the properties of the gamma function that 
\begin{equation*}
\binom{\nu+n-1}{n-1} = \frac{\Gamma(\nu+n)}{\Gamma(n)\Gamma(\nu+1)}>0, \quad \binom{\nu+n-1}{n} = \frac{\Gamma(\nu+n)}{\Gamma(n+1)\Gamma(\nu)}<0,
\end{equation*}
for all~$n\in\ZZ_+$. The result then immediately follows from~\eqref{eq_besselrec} by induction. 
\hfill\end{proof}

In the following, we use $\stirlingone{n}{k}$ for unsigned Stirling numbers of the first kind and $\stirlingtwo{n}{k}$ for Stirling numbers of the second kind. They are defined recursively through
\begin{equation}\label{eq_stirec}
\stirlingone{n+1}{k} = n \stirlingone{n}{k} + \stirlingone{n}{k-1} \quad \text{and} \quad \stirlingtwo{n+1}{k} = k \stirlingtwo{n}{k} + \stirlingtwo{n}{k-1},
\end{equation}
for $n,k\in\ZZ$, with initial conditions
\begin{equation*}
\stirlingone{0}{0} = \stirlingtwo{0}{0} = 1, \,\quad \stirlingone{n}{0} = \stirlingone{0}{n}= \stirlingtwo{n}{0} = \stirlingtwo{0}{n} = 0, \quad n\neq 0.
\end{equation*}
The combinatorial interpretation of these numbers is that $\stirlingone{n}{k}$ counts the number of permutations of $n$ elements with $k$ disjoint cycles, whereas $\stirlingtwo{n}{k}$ is the number of ways to partition a set of $n$ elements into $k$ nonempty subsets. 
The notation for Stirling numbers varies between different authors; we use the notation recommended in~\cite{Knuth}.

From the recursion in~\eqref{eq_besselrec} we derive the following explicit expression for the moments of~$A_1$. 

\begin{theorem}[Skew two-sided Bessel process]\label{thm_besselmom}
For any $n\geq 1$,
\begin{equation}\label{eq_besselmom}
\EE_0(A_1^n) = \sum_{k=0}^{n-1}\sum_{j=0}^{k} \frac{ (-1)^{j} j!}{(n-1)!} \stirlingone{n}{k+1} \stirlingtwo{k+1}{j+1} \nu^k \beta^{j+1}.
\end{equation}
\end{theorem}

In the proof of Theorem~\ref{thm_besselmom} we need the following result. 

\begin{lemma}\label{lemma_stirling}
For any $n, m, l \in\NN$, 
\begin{equation}\label{eq_stirling}
\stirlingone{n+1}{l+m+1} \binom{l+m}{l} = \sum_{k=l}^{n-m} \stirlingone{k+1}{l+1} \stirlingone{n-k}{m} \binom{n}{k}.
\end{equation}
\end{lemma}

\begin{proof}
Using the well known and very much similar identity \cite[Eq.~(6.29)]{Graham}
\begin{equation}\label{eq_stirlingid}
\stirlingone{n}{l+m} \binom{l+m}{l} = \sum_{k=l}^{n-m} \stirlingone{k}{l} \stirlingone{n-k}{m} \binom{n}{k},
\end{equation}
we prove~\eqref{eq_stirling} by induction. 
It is easy to verify that~\eqref{eq_stirling} holds for any $m,l\in\NN$ when $n=0$ (both sides are zero except when $m=l=0$, in which case both sides are equal to~1). Assume that~\eqref{eq_stirling} holds for~$n=N$ and all~$m,l\in\NN$. 
Now let $n=N+1$ and let~$m$ and~$l$ be arbitrary numbers in $\NN$. 
Using the recursion in~\eqref{eq_stirec} we get that
\begin{align*}
&\stirlingone{N+2}{l+m+1} \binom{l+m}{l} \\
&\hspace{0.5em} = (N+1)\stirlingone{N+1}{l+m+1} \binom{l+m}{l} + \stirlingone{N+1}{l+m} \binom{l+m}{l} \\
&\hspace{0.5em} = (N+1)\sum_{k=l}^{N-m} \stirlingone{k+1}{l+1} \stirlingone{N-k}{m} \binom{N}{k} + \!\sum_{k=l}^{N+1-m} \stirlingone{k}{l} \stirlingone{N+1-k}{m} \binom{N+1}{k} \\
&\hspace{0.5em} = \stirlingone{N+1-l}{m} \binom{N+1}{l} + \!\sum_{k=l+1}^{N+1-m} \left( k\stirlingone{k}{l+1} + \stirlingone{k}{l} \right) \stirlingone{N+1-k}{m} \binom{N+1}{k} \\
&\hspace{0.5em} = \sum_{k=l}^{N+1-m} \stirlingone{k+1}{l+1} \stirlingone{N+1-k}{m} \binom{N+1}{k}, 
\end{align*}
where in the second step we have used both the induction assumption and~\eqref{eq_stirlingid}. Thus, \eqref{eq_stirling}~holds also for~$n=N+1$ and all~$m,l\in\NN$. The result follows by induction. 
\hfill\end{proof}

\begin{proof}[Proof of Theorem~\ref{thm_besselmom}]
The result is proved by induction from Theorem~\ref{thm_besselrec}. The identity~\eqref{eq_besselmom} holds for~$n=1$, since $\EE_0(A_1)=\beta$. Assume that~\eqref{eq_besselmom} holds for all~$n\in\{1,2,\dotsc,N\}$. 
We wish to prove that it then also holds for~$n=N+1$, that is,
\begin{align}\label{eq_exleft}
\EE_0(A_1^{N+1}) &= \sum_{k=0}^{N}\sum_{j=0}^{k} \frac{ (-1)^{j} j!}{N!} \stirlingone{N+1}{k+1} \stirlingtwo{k+1}{j+1} \nu^k \beta^{j+1} \nonumber\\
&= \frac{\beta}{N!} \sum_{k=0}^{N} \stirlingone{N+1}{k+1} \nu^k + \sum_{k=1}^{N}\sum_{j=1}^{k} \frac{ (-1)^{j} j!}{N!} \stirlingone{N+1}{k+1} \stirlingtwo{k+1}{j+1} \nu^k \beta^{j+1} \nonumber\\
&= \beta \binom{\nu +N}{N} - \frac{1}{N!} \sum_{k=1}^{N}\sum_{j=1}^{k} j! \stirlingone{N+1}{k+1} \stirlingtwo{k+1}{j+1} \nu^k (-\beta)^{j+1}. 
\end{align}
Here the last step follows from the identity \cite[Eq.~(6.11)]{Graham}
\begin{equation*}
\sum_{k=0}^{n} \stirlingone{n}{k} x^k = x(x+1)\dotsm(x+n-1) = \frac{\Gamma(x+n)}{\Gamma(x)}, 
\end{equation*}
which gives that
\begin{equation}\label{eq_stirbin}
\binom{\nu+n}{n} = \frac{1}{n!} \frac{\Gamma(\nu+n+1)}{\Gamma(\nu+1)} = \frac{1}{n!} \cdot \frac{1}{\nu} \sum_{k=1}^{n+1} \stirlingone{n+1}{k} \nu^k = \frac{1}{n!} \sum_{k=0}^{n} \stirlingone{n+1}{k+1} \nu^k. 
\end{equation}
On the other hand, by the recursive formula~\eqref{eq_besselrec} we have
\begin{equation}\label{eq_exright}
\EE_0(A_1^{N+1}) = \beta \binom{\nu+N}{N} - \beta \sum_{m=1}^{N} \binom{\nu+N-m}{N-m+1} \EE_0(A_1^{m}).
\end{equation}
Applying~\eqref{eq_stirbin}, the binomial coefficient in the summand can be rewritten as
\begin{align*}
\binom{\nu+N-m}{N-m+1} &= \frac{\nu}{N-m+1} \binom{\nu+N-m}{N-m} \\
&= \frac{1}{(N-m+1)!} \sum_{k=m}^{N} \stirlingone{N-m+1}{N-k+1} \nu^{N-k+1},
\end{align*}
and using the induction assumption it follows that
\begin{align*}
& \beta \sum_{m=1}^{N} \binom{\nu+N-m}{N-m+1} \EE_0(A_1^{m}) \\
&\hspace{0.5em} = \sum_{m=1}^{N} \binom{\nu+N-m}{N-m+1} \sum_{i=1}^{m}\sum_{j=1}^{i} \frac{(j-1)!}{(m-1)!} \stirlingone{m}{i} \stirlingtwo{i}{j} \nu^{i-1} (-\beta)^{j+1} \\
&\hspace{0.5em} = \sum_{m=1}^{N}\sum_{i=1}^{m}\sum_{j=1}^{i} \sum_{k=m}^N \frac{(j-1)!}{N!} \binom{N}{m-1} \stirlingone{N-m+1}{N-k+1}\stirlingone{m}{i} \stirlingtwo{i}{j} \nu^{N-k+i} (-\beta)^{j+1} \\
&\hspace{0.5em} = \frac{1}{N!} \sum_{k=1}^{N}\sum_{j=1}^{k} \nu^{k} (-\beta)^{j+1} (j-1)! \sum_{i=j}^{k} \sum_{m=i}^{N-k+i} \binom{N}{m-1} \stirlingone{N-m+1}{k-i+1}\stirlingone{m}{i} \stirlingtwo{i}{j}. 
\end{align*}
When this is inserted into~\eqref{eq_exright}, the result should be equivalent to~\eqref{eq_exleft}. After subtracting the identical first terms we are left with polynomials in $\nu$ and $\beta$ on both sides, and thus it suffices to show that the coefficients are equal. Indeed, 
\begin{align*}
& (j-1)! \sum_{i=j}^{k} \sum_{m=i}^{N-k+i} \binom{N}{m-1} \stirlingone{N-m+1}{k-i+1}\stirlingone{m}{i} \stirlingtwo{i}{j} \\
&\hspace{4em} = (j-1)! \sum_{i=j}^{k} \stirlingtwo{i}{j} \sum_{m=i-1}^{N-(k-i+1)} \binom{N}{m} \stirlingone{N-m}{k-i+1}\stirlingone{m+1}{i} \\
&\hspace{4em} = (j-1)! \sum_{i=j}^{k} \stirlingtwo{i}{j} \stirlingone{N+1}{k+1} \binom{k}{i-1} \\
&\hspace{4em} = j! \stirlingone{N+1}{k+1} \stirlingtwo{k+1}{j+1},
\end{align*}
where the second step follows from Lemma~\ref{lemma_stirling}. The last step holds since
\begin{align*}
\sum_{i=j}^{k} \stirlingtwo{i}{j} \binom{k}{i-1} &= \sum_{i=j}^{k} \stirlingtwo{i}{j} \left( \binom{k+1}{i} - \binom{k}{i} \right) \\
&= \sum_{i=j}^{k+1} \stirlingtwo{i}{j} \binom{k+1}{i} - \stirlingtwo{k+1}{j} - \sum_{i=j}^{k} \stirlingtwo{i}{j} \binom{k}{i} \\
&= \stirlingtwo{k+2}{j+1} - \stirlingtwo{k+1}{j} - \stirlingtwo{k+1}{j+1} \\
&= (j+1) \stirlingtwo{k+1}{j+1} - \stirlingtwo{k+1}{j+1} \\
&= j \stirlingtwo{k+1}{j+1},
\end{align*}
using both~\eqref{eq_stirec} and the identity \cite[Eq.~(6.15)]{Graham}
\begin{equation*}
\sum_{i=j}^{k} \stirlingtwo{i}{j} \binom{k}{i} = \stirlingtwo{k+1}{j+1}.
\end{equation*}
This completes the proof for~$n=N+1$. By induction, \eqref{eq_besselmom}~holds for all~$n\geq 1$. 
\hfill\end{proof}

\begin {remark}
The distribution of~$A_1$ for a skew two-sided Bessel process starting from~0 has been characterized in~\cite{Barlow}. 
The moments of $A_1$ can also be calculated numerically for particular values of~$\beta$ and~$\nu$ by integrating the known density~\cite{Watanabe}, for~$x\in(0,1)$,
\begin{equation}\label{eq_besdens}
f(x) = \frac{\frac{1}{\pi}\sin(-\nu \pi) \beta(1-\beta)(x(1-x))^{-\nu-1}}{\beta^2 (1-x)^{-2\nu}+(1-\beta)^2 x^{-2\nu} + 2\beta(1-\beta) (x(1-x))^{-\nu} \cos(-\nu\pi)}.
\end{equation}
We call the distribution induced by this density a Lamperti distribution, as it was first found in~\cite{Lamperti}. 
The result in~\eqref{eq_besselmom} does not, however, seem to be easily obtainable through analytic integration. For the special case of skew Brownian motion, see Remark~\ref{rem_skew}.
\end{remark}

\subsection{Skew Brownian motion}

Let now $(X_t)_{t\geq 0}$ be a skew Brownian motion, see Example~\ref{ex_skewBM}. This corresponds to a skew Bessel process with $\nu=-1/2$, and thus it follows from~\eqref{eq_besselD} that
\begin{equation}\label{SBMDA}
D_k(\lambda) = \beta\binom{k-\frac{3}{2}}{k} = \frac{-\beta}{2^{2k}(2k-1)} \binom{2k}{k},
\end{equation}
and the recursion in \eqref{eq_besselrec} becomes
\begin{align}\label{eq_skewrec}
\EE_0(A_1^n) &= \frac{\beta}{2^{2n-2}}\binom{2n-2}{n-1} + \beta \sum_{k=1}^{n-1} \EE_0(A_1^{n-k}) \frac{1}{2^{2k}(2k-1)} \binom{2k}{k}.
\end{align}
Since the recursion has already been solved for skew Bessel processes, the moments of~$A_1$ for skew Brownian motion can be obtained from Theorem~\ref{thm_besselmom}. 

\begin{theorem}[Skew Brownian motion]\label{thm_skewmom}
For any $n\geq 1$,
\begin{equation}\label{eq_skewmom}
\EE_0(A_1^n) = \sum_{k=0}^{n-1} \binom{n-1+k}{k} \frac{\beta^{n-k}}{2^{n+k-1}}.
\end{equation}
\end{theorem}

\begin{proof}
Substituting $\nu=-1/2$ in~\eqref{eq_besselmom} yields
\begin{align*}
\EE_0(A_1^n) &= \sum_{i=1}^{n}\sum_{k=1}^{i} \frac{ (-1)^{k-1} (k-1)!}{(n-1)!} \stirlingone{n}{i} \stirlingtwo{i}{k} \left(-\frac{1}{2}\right)^{i-1} \beta^{k} \\
&= \sum_{k=1}^{n} \frac{ (-1)^{k-n} \beta^{k} (k-1)!}{2^{n-1}(n-1)!} \sum_{i=k}^{n} \stirlingone{n}{i} \stirlingtwo{i}{k} (-2)^{n-i} \\
&= \sum_{k=1}^{n} \frac{ \beta^{k}}{2^{2n-k-1}}\binom{2n-k-1}{n-k},
\end{align*}
where in the last step the identity \cite[Eq.~(18)]{YangQiao} 
\begin{equation}\label{eq_bss}
\sum_{i=k}^{n} \stirlingone{n}{i}\stirlingtwo{i}{k} (-2)^{n-i} = (-1)^{n-k}\frac{(2n-k-1)!}{2^{n-k} (k-1)!(n-k)!}
\end{equation}
has been applied. A change of order of summation now gives the result. 
\hfill\end{proof}

\begin{remark}
The result in Theorem~\ref{thm_skewmom} was first proved by solving the recursive formula~\eqref{eq_skewrec} in a different way than in the proof of Theorem~\ref{thm_besselmom}. However, later the authors became aware that the identity \eqref{eq_bss} is already found in \cite{YangQiao}, and thus the result follows from \eqref{eq_besselmom}, as shown above. The earlier, alternative method is presented with additional comments in~\cite{Stenlund}. 
\end{remark}

\begin{corollary}[Standard Brownian motion]\label{cor11}
For $\beta=1/2$ and $n\geq 1$ it holds that
\begin{equation*}
\EE_0(A_1^n) = \frac{1}{2^{2n}} \binom{2n}{n},
\end{equation*}
and, hence, $A_1$ has the arcsine law. 
\end{corollary}

\begin{proof}
The result follows from~\eqref{eq_skewmom} when inserting $\beta=1/2$ and applying~\eqref{eq_Gsum}. 
Since $A_1$ is bounded its moments determine the distribution uniquely, and we recover L\'{e}vy's arcsine law for~$A_1$. 
\hfill\end{proof}

\begin{remark}\label{rem_skew}
The distribution function of~$A_1$ for skew Brownian motion is given in~\cite[Eq.~(2.5), p.~158]{Watanabe} as
\begin{equation*}
\PP_0(A_1\leq x) = \frac{2}{\pi} \arcsin \left( \sqrt{\frac{x}{x+(\beta/(1-\beta))^2(1-x)}} \right), \quad x\in[0,1].
\end{equation*}
This yields the density (see~\cite[p.~782]{WYY} and \cite[p.~196]{Appuhamillage})
\begin{equation*}
f_{A_1}(x) = \frac{\beta(1-\beta)}{\pi \sqrt{x(1-x)} (\beta^2 +x(1-2\beta))}, \quad x\in(0,1),
\end{equation*}
which corresponds to~\eqref{eq_besdens} with $\nu=-1/2$. From this density it is possible to calculate the moments of~$A_1$ by integration. However, using this approach we have not been able to obtain the expression in~\eqref{eq_skewmom}.
\end{remark}

\subsection{Oscillating Brownian motion}

Let $(\widetilde X_t)_{t\geq 0}$ be an oscillating Brownian motion, see Example~\ref{ex_osc}. From~\eqref{oscscaling} it is easily seen that the law of~$A_1$ is the same as the corresponding law for a skew Brownian motion. However, here we wish to demonstrate the use of Theorem~\ref{thm_rec} and therefore give a proof of the following result based on formula~\eqref{eq_E0A1n} therein.

\begin{theorem}[Oscillating Brownian motion]
For any $n\geq 1$,
\begin{equation*}
\EE_0(A_1^n) = \sum_{k=0}^{n-1} \binom{n-1+k}{k} \frac{1}{2^{n+k-1}} \left( \frac{\sigma_{-}}{\sigma_{+}+\sigma_{-}} \right)^{n-k}.
\end{equation*}
\end{theorem}

\begin{proof}
For oscillating Brownian motion,
\begin{equation*}
\widehat{f}(x;\lambda) = \EE_x(\ee^{-\lambda H_0}) = \ee^{-\frac{x\sqrt{2\lambda}}{\sigma_+}}, \quad x\geq 0. 
\end{equation*}
Recalling~\eqref{eq_oscG}, Equation~\eqref{eq_Ck} then becomes
\begin{align*}
\widetilde{D}_k(\lambda) 
&= \frac{(-\lambda)^k}{(k-1)!} \int_0^\infty \frac{\sigma_+ \sigma_-}{(\sigma_+ + \sigma_-)\sqrt{2\lambda}} \ee^{-\frac{y\sqrt{2\lambda}}{\sigma_+}} \frac{\diff^{k-1}}{\diff \lambda^{k-1}}\biggl( \ee^{-\frac{y\sqrt{2\lambda}}{\sigma_+}} \biggr) \frac{2}{\sigma_{+}^2} \diff y \\
&= \frac{2 \sigma_-}{\sigma_+ + \sigma_-} \frac{(-\lambda)^k}{(k-1)!} \int_0^\infty \frac{1}{2\sqrt{2\lambda}} \ee^{-z\sqrt{2\lambda}} \frac{\diff^{k-1}}{\diff \lambda^{k-1}} \bigl( \ee^{-z\sqrt{2\lambda}} \bigr) 2 \diff z \\
&= \frac{\sigma_-}{\sigma_+ + \sigma_-} \frac{1}{\beta} D_k(\lambda),
\end{align*}
where $D_k(\lambda)$ is as for skew Brownian motion. In other words, by~\eqref{SBMDA},
\begin{equation*}
\widetilde{D}_k(\lambda) = \left(\frac{\sigma_-}{\sigma_+ + \sigma_-}\right) \frac{-1}{2^{2k}(2k-1)} \binom{2k}{k}. 
\end{equation*}
Since
\begin{equation*}
\EE_0(A_1) = \lambda\int_0^\infty \widetilde{G}_\lambda(0,y) m_\nu(\diff y) = \frac{\sigma_-}{\sigma_+ + \sigma_-},
\end{equation*}
we thus arrive at the same recursion as in~\eqref{eq_skewrec}, except with $\sigma_{-}/(\sigma_{+}+\sigma_{-})$ instead of~$\beta$. Thus, the result follows from Theorem~\ref{thm_skewmom}. 
\hfill\end{proof}

\subsection{Brownian spider}

For a positive integer $n$, let $I_1,\dotsc,I_n$ be half lines in $\RR^2$ meeting at the origin. Such a configuration can be seen as a graph $G$, say, with one vertex and $n$ infinite edges. A Brownian spider, also called Walsh's Brownian motion on a finite number of rays, see~\cite{Barlow2,Barlow,Lejay,Vakeroudis,Walsh}, is a diffusion on $G$ such that when away from the origin on a ray and before hitting the origin, it behaves on that ray like an ordinary one-dimensional Brownian motion, but as the process reaches the origin one of the half-lines is chosen randomly for its ``next'' excursion. A rigorous construction of the probability measure governing a Brownian spider can be done using the excursion theory, see~\cite{Barlow}. 

Let $(X_t)_{t\geq 0}$ denote a Brownian spider living on $G$, and for every $i\in\{1,2,...n\}$ let $p_i$ be the probability for choosing half-line $I_i$ when at the origin. For any subset $\mathcal{I}\subseteq \{1,2,\dotsc,n\}$ consider the occupation time of~$X$ on $\{I_i:\, i\in \mathcal{I}\},$ i.e., 
\begin{equation*}
A_t^\mathcal{I} = \Leb \{s\in[0,t]: X_s\in \bigcup_{i\in \mathcal{I}} I_i\}.
\end{equation*}
From the excursion theoretical construction of the Brownian spider it can be deduced that $A_t^\mathcal{I}$ has the same law as the occupation time on $\RR_+$ for a skew Brownian motion with the skewness parameter $\beta := \sum_{i\in \mathcal{I}} p_i$. Thus, we deduce the following result from~\eqref{eq_skewmom}. 

\begin{theorem}[Brownian spider]
For any $n\geq 1$,
\begin{equation*}
\EE_0((A_1^\mathcal{I})^n) = \sum_{k=0}^{n-1} \binom{n-1+k}{k} \frac{1}{2^{n+k-1}} \bigg( \sum_{i\in \mathcal{I}} p_i \bigg)^{n-k}.
\end{equation*}
\end{theorem}

We refer also to the recent paper \cite{Yano} for results concerning distributions of occupation times for diffusions on a multiray.

\subsection{Sticky Brownian motion}

In this section we highlight the use of the moment formula for a process that does not have the scaling property. Moreover, we wish to understand how the presence of a sticky point affects the formula. To this end, let $X=(X_t)_{t\geq 0}$ be a Brownian motion sticky at~0 with stickyness parameter $\gamma>0$, see Example~\ref{ex_sticky}. 
Since $X$ behaves like a standard Brownian motion on excursions from~0, the Laplace transform of the first hitting time of~0 when $X_0=x>0$ is as for standard Brownian motion, i.e., 
\begin{equation*}
\widehat{f}(x;\lambda) :=\EE_x\left(\ee^{-\lambda H_0}\right)= \ee^{-x\sqrt{2\lambda}}. 
\end{equation*}
In this section, besides $A^X_t$ defined in~\eqref{def_At}, we also consider 
\begin{equation*}
B_t^X := \Leb \{s\in[0,t]:X_s>0\}.
\end{equation*}
Clearly,
\begin{equation}\label{A_B}
A_t =B_t + \Leb \{s\in[0,t]:X_s=0\},
\end{equation} 
where $A_t$ and $B_t$ are used as notation for~$A^X_t$ and $B^X_t,$ respectively. 
Since~0 is a sticky point the second term on the right hand side of~\eqref{A_B} is strictly positive for all~$t>0$~a.s. (if~$X$ starts at~0). We use here the notation $D^A_n(\lambda)$ instead of~$D_n(\lambda)$ as defined in~\eqref{eq_Ck}:
\begin{equation*}
D^A_n(\lambda) := \frac{(-\lambda)^n}{(n-1)!} \intco{0}{\infty} G_\lambda(0,y) \widehat{f}_\lambda^{(n-1)}(y;\lambda) m(\diff y).
\end{equation*} 
Using the explicit form~\eqref{green0} for the Green kernel yields for~$n=1$
\begin{align*}
D^A_1(\lambda) &= -2\lambda \left(\int_0^\infty G_\lambda(0,y) \widehat{f}_\lambda(y;\lambda) \diff y + \gamma\, G_\lambda(0,0) \widehat{f}_\lambda(0;\lambda) \right) \\
&=\frac{-2\lambda}{2 \sqrt{2\lambda}+2\lambda\gamma} \left(\frac{1}{2\sqrt{2\lambda}}+\gamma\right) \\
&= - H(\lambda)\left(\frac{1}{2}+\gamma\sqrt{2\lambda}\right),
\end{align*} 
where 
\begin{equation}\label{Hll}
H(\lambda):=\left(2 + \gamma\sqrt{2\lambda}\right)^{-1}.
\end{equation}
Comparing the expressions for the Green kernels in~\eqref{SBMG} and~\eqref{green0} and recalling~\eqref{SBMDA}, as well as Lemma~\ref{lemma_f0}, it is seen that for~$n\geq 2$
\begin{align*}
D^A_n(\lambda) &= \frac{2(-\lambda)^n}{(n-1)!} \int_{(0,\infty)} G_\lambda(0,y) \widehat{f}_\lambda^{(n-1)}(y;\lambda) \diff y\\
& = \frac{-1}{2^{2n}(2n-1)} \binom{2n}{n}\, \frac{\sqrt{2\lambda}}{2\sqrt{2\lambda} +2\gamma\lambda}\\
& =- H(\lambda)\, T_n
\end{align*}
where 
\begin{equation}\label{Tnn}
T_n:= \frac{1}{2^{2n}(2n-1)} \binom{2n}{n} 
\end{equation}
The values for $D^B_n(\lambda)$ are obtained in the same way, and we conclude the discussion above in the following result. 

\begin{proposition}
It holds that
\begin{equation*}
 D^A_1(\lambda) = - \frac{H(\lambda)}{2}\left(1+2\gamma\sqrt{2\lambda}\right),\qquad D^B_1(\lambda) = - \frac{H(\lambda)}{2} = - H(\lambda)\, T_1, 
\end{equation*}
and, for~$n\geq2$,
\begin{equation*}
D^A_n(\lambda) = D^B_n(\lambda) = - H(\lambda)\, T_n,
\end{equation*} 
with $H(\lambda)$ and $T_n$ given in~\eqref{Hll} and~\eqref{Tnn}, respectively. 
\end{proposition}

Recall the recursive equation~\eqref{eq_Ahat_rec} for the Laplace transforms of the moments of~$B_t$ (an analogous formula holds for~$A_t$):
\begin{equation}\label{Amom}
\widehat B_n(\lambda) := \widehat{B}_0(\lambda;n) = \frac{n!}{\lambda^{n-1}} \widehat B_1(\lambda) + \frac{n!}{\lambda^{n+1}}\sum_{k=1}^{n-1} \left( 1 - \frac{\lambda^{n-k+1}}{(n-k)!} \widehat B_{n-k}(\lambda) \right) D^B_{k}(\lambda).
\end{equation}
Next we show that this recursive equation can be solved similarly as in the case of skew Brownian motion. The corresponding equation for the Laplace transforms of~$A_t$ does not seem to allow such a simple solution. 
 
\begin{proposition}
For $n=1,2,\dotsc$,
\begin{equation}\label{Amom03}
\widehat B_n(\lambda) = \frac{n!}{\lambda^{n+1}} \sum_{k=0}^{n-1} \binom{n-1+k}{k} \frac{1}{2^{n+k-1}} \left( \frac{1}{2+\gamma\sqrt{2\lambda}} \right)^{n-k}.
\end{equation}
\end{proposition}

\begin{proof}
Introducing $U_k(\lambda) := \lambda^{k+1}\widehat{B}_k(\lambda)/k!$ as in Remark~\ref{rem_U}, equation~\eqref{Amom} can be rewritten as
 \begin{align}\label{Amom01}
U_n(\lambda) &= U_1(\lambda) + \sum_{k=1}^{n-1} (1 - U_{n-k}(\lambda)) (- H(\lambda)T_k) \nonumber\\
&= H(\lambda) \Biggl( 1 - \sum_{k=1}^{n-1} (1 - U_{n-k}(\lambda)) T_k \Biggr),
\end{align}
since
\begin{equation*}
U_1(\lambda) = \lambda^2 \widehat{B}_1(\lambda) = \lambda \int_0^\infty G_\lambda(0,y) m(\diff y) = H(\lambda).
\end{equation*} 
Clearly, \eqref{Amom01}~is of the same form as~\eqref{eq_skewrec} and, consequently, by Theorem~\ref{thm_skewmom}, 
\begin{equation}\label{eq_U1}
U_n(\lambda) = \sum_{k=0}^{n-1} \binom{n-1+k}{k} \frac{(H(\lambda))^{n-k}}{2^{n+k-1}},
\end{equation}
which is the same as
\begin{equation*}
\widehat B_n(\lambda) = \frac{n!}{\lambda^{n+1}} \sum_{k=0}^{n-1} \binom{n-1+k}{k} \frac{(H(\lambda))^{n-k}}{2^{n+k-1}},
\end{equation*}
and the claimed formula~\eqref{Amom03} follows. 
\hfill\end{proof}

Consider now a regular diffusion $X=(X_t)_{t\geq 0}$ as introduced in Section~\ref{section_prelim}. In~\cite[p.~161]{Watanabe} are given necessary and sufficient conditions in terms of the speed measure of~$X$ that ensure that $A^X_t/t$ converges in distribution to a random variable $\xi$ which is Lamperti-distributed, i.e., the density of~$\xi$ is given by~\eqref{eq_besdens}. It is a fairly simple matter to check these conditions for a sticky Brownian motion with the limiting random variable being then arcsine-distributed. We conclude the paper by showing that this result is also easily obtained for both $A_t/t$ and $B_t/t$ using the recursive equation. Notice that in this case the convergence of the moments is equivalent to the convergence in distribution. 

\begin{proposition}
For $n=1,2,\dotsc$ 
\begin{equation}\label{lim11}
\lim_ {\lambda\to 0} \lambda^{n+1}\widehat A_n(\lambda)) =\lim_ {\lambda\to 0} \lambda^{n+1}\widehat B_n(\lambda)) = \frac{n!}{2^{2n}} \binom{2n}{n}
\end{equation}
and 
\begin{equation}\label{lim12}
\lim_{t\to\infty} t^{-n}\EE_0(A_t^n) = \lim_{t\to\infty} t^{-n}\EE_0(B_t^n) = \frac{1}{2^{2n}} \binom{2n}{n}.
\end{equation}
\end{proposition}

\begin{proof}
We prove~\eqref{lim11} for~$B_t$ with induction from~\eqref{Amom}. An analogous reasoning is valid for~$A_t$ and, hence, the details are omitted. The claim for~$B_t$ holds for~$n=1$, as is easily seen from~\eqref{eq_U1}. Multiplying~\eqref{Amom01} with $n!$ yields 
\begin{equation*}
\lambda^{n+1}\widehat B_n(\lambda) = {n!}\left( H(\lambda) - H(\lambda) \sum_{k=1}^{n-1} \left( 1 - \frac{\lambda^{n-k+1}}{(n-k)!} \widehat B_{n-k}(\lambda) \right) T_{k}\right). 
\end{equation*}
By the induction assumption the limit of the right hand side exists as $\lambda\to 0$, implying
\begin{equation}\label{Amom2}
\lim_{\lambda\to 0} \lambda^{n+1}\widehat B_n(\lambda) = {n!}\left(\frac{1}{2} - \frac{1}{2} \sum_{k=1}^{n-1} 
\left( 1 - \frac{1}{2^{2k}} \binom{2k}{k} \right) T_{n-k}\right). 
\end{equation}
From equation~\eqref{eq_skewrec} and Corollary~\ref{cor11} it is seen that the expression in the outer parenthesis of~\eqref{Amom2} is as claimed in~\eqref{lim11}. The statement~\eqref{lim12} follows by evoking the Tauberian theorem presented in~\cite[p~423]{Feller}, which is applicable since $t\mapsto \EE_0(B_t^n)$ is increasing. 
\hfill\end{proof}

\appendix
\section[Appendix]{Appendix: Proof of Lemma~\ref{lemma_f0}}

Our proof is based on the spectral representation of the $\PP_x$-density~$f$ of~$H_0$. To fix ideas let $x>0$. The main source of the proof is \cite{Salminen}, where also further references and results can be found. The key fact presented in~\cite[Prop.~3.3]{Salminen} is that $f$ has the representation
\begin{equation*}
f(x;t) = \int_0^\infty \ee^{-\gamma t} C(x;\gamma) \widehat{\Delta}(\diff \gamma),
\end{equation*}
where 
\begin{equation*}
C(x;\gamma) := \sum_{n=0}^\infty (-\gamma)^n c_n(x),
\end{equation*}
with
\begin{equation}
\label{cn}
c_0(x): = S(x), \quad
c_{n+1}(x) := \int_0^x \diff S(y) \int_0^y m(\diff z) c_n(z).
\end{equation}
Recall that $m$ is the speed measure and $S$ is the scale function such that $S(0)=0$. The notation $\widehat{\Delta}$ stands for a $\sigma$-finite measure on $(0,\infty)$ such that
\begin{equation}
\label{delta}
\int_0^\infty \frac{\widehat{\Delta}(\diff z)}{z(z+1)} < \infty.
\end{equation}
From~\eqref{cn} it is seen that $x\mapsto c_n(x), n=1,2,\dotsc$ are positive, continuous and increasing with $\lim_{x\downarrow 0} c_n(x) = 0$. Hence, 
\begin{align}
\lim_{x\downarrow 0} C(x;\gamma) &= \lim_{x\downarrow 0} \sum_{n=0}^\infty (-\gamma)^n c_n(x) 
= \sum_{n=0}^\infty (-\gamma)^n \lim_{x\downarrow 0} c_n(x) = 0, \label{eq_Cto0}
\end{align}
where in the second step the use of dominated convergence is justified by the estimate in~\cite[Eq.~(3.12)]{Salminen}, i.e., for~$x<1$
\begin{align*}
\Bigl|\sum_{n=0}^\infty (-\gamma)^n c_n(x)\Bigr| \leq \sum_{n=0}^\infty \gamma^n c_n(1)
&\leq \sum_{n=0}^\infty \gamma^n \frac{1}{n!}\, S(1) \left( \int_0^1 m(y) \diff S(y) \right)^n \\
&= S(1) \exp\big({\gamma \int_0^1 m(y) \diff S(y)}\big) < \infty. 
\end{align*}
Consider now for~$k=1,2,\dots$ 
\begin{align*}
\bigl| \widehat{f}_\lambda^{(k)}(x;\lambda) \bigr| &= \EE_x(H_0^k \ee^{-\lambda H_0}) \\
&= \int_0^\infty t^k \ee^{-\lambda t} f(x;t) \diff t \\
&= \int_0^\infty t^k \ee^{-\lambda t} \left( \int_0^\infty \ee^{-\gamma t} C(x;\gamma) \widehat{\Delta}(\diff \gamma) \right) \diff t .
\end{align*}
Letting here $x\downarrow 0$ and evoking~\eqref{eq_Cto0} yields the claim in Lemma 2, i.e.,~\eqref{eq_f0} -- once we have checked that the limit can be taken inside the integrals. For this, recall for~$x<1$ the following estimate from \cite[p.~1972]{Salminen}:
\begin{equation*}
|\ee^{-\gamma t} C(x;\gamma)| \leq S(x) \ee^{-\gamma t/2} \leq S(1) \ee^{-\gamma t/2}.
\end{equation*}
Consequently, by Fubini's theorem
\begin{align}
\nonumber \biggl| \int_0^\infty t^k \ee^{-\lambda t} \biggl( \int_0^\infty \ee^{-\gamma t} C(x;\gamma) \widehat{\Delta}(\diff \gamma) \biggr) \diff t \biggr| 
&\leq \int_0^\infty t^k \ee^{-\lambda t} \biggl( \int_0^\infty S(1) \ee^{-\gamma t/2} \widehat{\Delta}(\diff \gamma) \biggr) \diff t \\
&= S(1) \int_0^\infty \biggl( \int_0^\infty t^k \ee^{-(2\lambda+\gamma )t/2} \diff t \biggr)\widehat{\Delta}(\diff \gamma) \nonumber\\
&= S(1) \int_0^\infty \frac{2^{k+1}\, k!}{(2\lambda+\gamma)^{k+1}}\widehat{\Delta}(\diff \gamma). \label{eq_dominated}
\end{align}
By~\eqref{delta} the right hand side of~\eqref{eq_dominated} is finite. Hence, by dominated convergence, 
\begin{equation*}
\lim_{x\downarrow 0} \bigl| \widehat{f}_\lambda^{(k)}(x;\lambda) \bigr| = \int_0^\infty t^k \ee^{-\lambda t} \left( \int_0^\infty \ee^{-\gamma t} \lim_{x\downarrow 0} C(x;\gamma) \widehat{\Delta}(\diff \gamma) \right) \diff t = 0,
\end{equation*}
as claimed in~\eqref{eq_f0}.

\subsection*{Acknowledgments} 
We thank Ernesto Mordecki for a discussion which triggered the research presented in this paper. The research of D.~Stenlund was supported in part by a grant from the Magnus Ehrnrooth foundation.

\end{document}